\documentclass[reqno, a4paper, 11pt]{amsart}

\usepackage{setspace}
\usepackage{fullpage}
\usepackage[backref,colorlinks]{hyperref}
\usepackage{amssymb,bm}
\usepackage[foot]{amsaddr}
\usepackage[abbrev,msc-links]{amsrefs}
\usepackage[utf8]{inputenc}
\usepackage[shortlabels]{enumitem}
\usepackage{graphicx}
\usepackage{mathtools}
\usepackage{lineno}
\usepackage[framemethod=TikZ]{mdframed}
\usepackage{cases}

\theoremstyle{plain}
\newtheorem{theorem}{Theorem}[section]
\newtheorem{fact}[theorem]{Fact}
\newtheorem{prop}[theorem]{Proposition}

\newtheorem{cor}[theorem]{Corollary}
\newtheorem{lemma}[theorem]{Lemma}

\newtheorem*{clm*}{Claim}

\theoremstyle{definition}
\newtheorem{rem}[theorem]{Remark}
\newtheorem{Def}[theorem]{Definition}
\newtheorem{example}[theorem]{Example}
\newtheorem{qn}[theorem]{Question}
\newtheorem{conj}[theorem]{Conjecture}
\newtheorem{obs}[theorem]{Observation}
\newtheorem{prob}[theorem]{Problem}

\newmdtheoremenv[skipabove=2mm,skipbelow=4mm]{const}[theorem]{Construction}

\numberwithin{equation}{section}

\DeclareMathOperator{\ex}{ex}
\DeclareMathOperator{\ws}{wsat}
\DeclareMathOperator{\sat}{sat}

\DeclareMathOperator{\dist}{dist}

\DeclarePairedDelimiter\ceil{\lceil}{\rceil}
\DeclarePairedDelimiter\floor{\lfloor}{\rfloor}

\DeclarePairedDelimiter{\final}{\langle}{\rangle}
\DeclarePairedDelimiter\Set{\lbrace}{\rbrace}

\newcommand\N{\mathbb N}
\newcommand\Z{\mathbb Z}

\newcommand\cH{\mathcal{H}}

\let\phi=\varphi
\let\emptyset=\varnothing

\usepackage{euflag}

\let\oldtocsection=\tocsection

\let\oldtocsubsection=\tocsubsection

\let\oldtocsubsubsection=\tocsubsubsection

\renewcommand{\tocsection}[2]{\hspace{0em}\oldtocsection{#1}{#2}}
\renewcommand{\tocsubsection}[2]{\hspace{1em}\oldtocsubsection{#1}{#2}}
\renewcommand{\tocsubsubsection}[2]{\hspace{2em}\oldtocsubsubsection{#1}{#2}}

\let \phi=\varphi

\newcommand\tw{{\mathrm{tw}}}

\begin{document}

\title[Slow graph bootstrap percolation]{Graph bootstrap percolation \\--- a discovery of slowness}

\date{\today}
\author{David Fabian}
\author{Patrick Morris $^{1,\ast}$}
\author{Tibor Szab\'o $^{2,\dagger}$}

\address{$^1$ Departament de Matem\`atiques, Universitat Polit\`ecnica de Catalunya (UPC), Barcelona, Spain.}
\address{$^2$ Institute of Mathematics, Freie Universit\"at Berlin, Germany}

\email{dfabian@mailbox.org, pmorrismaths@gmail.com,  szabo@math.fu-berlin.de}

\thanks{$^\ast$ Research supported by   the DFG Walter Benjamin program - project number 504502205, and by   the European Union's Horizon Europe   Marie Sk{\l}odowska-Curie grant RAND-COMB-DESIGN - project number
101106032 {\euflag}.}

\thanks{$^\dagger$ Research
	supported by the DFG under Germany’s Excellence Strategy - The Berlin Mathematics Research Center
	MATH+ (EXC-2046/1, project ID: 390685689).}

\begin{abstract}

Graph bootstrap percolation is a discrete-time process capturing the spread of a virus on the edges of $K_n$. Given an initial set $G\subseteq K_n$ of infected edges, the transmission of the virus is governed by a fixed graph $H$: in each round of the process any edge $e$ of $K_n$ that is the last uninfected edge in a copy of $H$ in $K_n$ gets infected as well. Once infected, edges remain infected forever.  The process was introduced by Bollob\'as in 1968 in the context of weak saturation and has since inspired a vast array of beautiful mathematics. The main focus of this survey is the extremal question of {\em how long} the infection process can last before stabilising. 
We give an exposition of our recent systematic study of this \textit{maximum running time} and the influence of the infection rule $H$. 
The topic turns out to possess  a wide variety of interesting behaviour, with connections to additive, extremal and probabilistic combinatorics. Along the way we encounter a number of  
surprises and attractive  open problems. 
\end{abstract}

\maketitle

\vspace{24mm}
{\centering
A survey prepared for the occasion of the \\ $31^{\text{st}}$ British Combinatorial Conference (BCC) 2026 \\ Cardiff University, Cardiff, Wales.\par}

\newpage
{\fontsize{12.5}{12.5}\selectfont \tableofcontents}
\newpage

 \section{Introduction} \label{sec:intro}

 The central notion of this survey is \textit{graph bootstrap percolation}, a term coined by  Balogh, Bollob\'as, and Morris \cite{balogh2012graph}, although the process itself appeared much earlier, in a different context, in work of Bollob\'as \cite{bollobas1968weakly}. 
 Given graphs $H$ and $G$, let $n_H(G)$ denote the number of copies of $H$ in $G$. The $H$\textit{-bootstrap percolation process}
  ($H$\emph{-process} for short)
on a graph $G$ is the sequence $(G_i)_{i\geq 0}$ of graphs defined by $G_0 := G$ and for $i\geq 1$, $V(G_i) := V(G)$ and 
     \begin{equation*}
     E(G_i) := E(G_{i-1}) \cup \left\{e\in\binom {V(G)} 2 \ : \ n_H\left(G_{i-1}+e\right)>n_H(G_{i-1})\right\}.
               \end{equation*} 
 Informally, one can imagine a virus spreading in discrete time steps on the edges of the complete graph on $V(G)$. Initially just the edges of $G$ are infected. The graph $H$ governs the spread of the virus: any edge $e$ of $K_n$ which is the last uninfected edge in a copy of $H$ in $K_n$ gets infected as well.
 Once infected, edges remain infected forever. We call $G$ the \emph{starting graph} and $H$ the {\em infection rule} of the process.

Note that for finite $G$, 
the $H$-process on $G$  stabilises and we refer to $\tau_H(G):=\min\{t\in \N: G_t=G_{t+1}\}$, i.e., the number of rounds it takes for that to happen, as the \emph{running time} of the $H$-process on $G$. 
 We define $G_\tau$ with $\tau=\tau_H(G)$ to be the \emph{final graph}  of the process, and denote it $\final{G}_H:=G_\tau$. If eventually all edges get infected, that is, the final graph $\final{G}_H$ is the clique $K_{v(G)}$, we say that the starting graph $G$ is {\em $H$-percolating}.

  \begin{example}[The triangle] 
  \label{ex:triangle}
  Let $H=K_3$ and $G$ be an arbitrary $n$-vertex graph. Note first that no two components of $G$ become connected during the $K_3$-process as two non-adjacent vertices become adjacent during a round of the $K_3$-process if and only if they have a common neighbour already. Now if $x$ and $y$ are in the same component in $G_{i-1}$ then in the next round their distance along a shortest path is reduced to roughly half: $\dist_{G_i}(x,y) = \lceil \frac{1}{2} \dist_{G_{i-1}}(x,y)\rceil$. In particular, after $\lceil \log_2 (\dist_G(x,y)) \rceil$ rounds their distance is one and they are adjacent. Thus we have that $\tau_H(G)=\max_{x,y}\lceil \log_2 (\dist_G(x,y)) \rceil$ with the maximum taken over pairs of vertices $x,y$ in the same connected component of $G$,  and $\final{G}_H$ is a union of disjoint cliques on the vertex sets of the connected components.
  \end{example}

\subsection{Weak saturation} The graph percolation process was introduced by Bollob\'as~\cite{bollobas1968weakly} in 1968 for cliques, under the name  \textit{weak saturation}, and has since inspired a vast array of beautiful mathematics. In our terminology, Bollob\'as was interested in the smallest possible number of infected edges that percolate in the $K_k$-process. Subsequently this extremal question turned out to be greatly influential, as it inspired one of the first applications of algebra in extremal combinatorics.  
For an arbitrary graph $H$, $\ws(n,H)$ denotes the minimum number of edges in an $H$-percolating graph on $n$ vertices. 
The concept is motivated by work of Erd\H{o}s, Hajnal and Moon \cite{erdos1964problem}, who determined the smallest number of edges on $n$ vertices which $K_k$-percolate in a {\em single round}.    
Graphs that are $H$-free and $H$-percolate in one time step  are called \textit{$H$-saturated}. The minimum number of edges in an $H$-saturated graph is denoted by $\sat(n,H)$, whilst the maximum number is the {\em Tur\'an number} $\ex(n,H)$ of $H$. Clearly $\ws(n,H) \leq \sat (n,H)$. 
Considering Example \ref{ex:triangle} we see that $\ws(n,K_3)=\sat(n,K_3)=n-1$ as any tree  on $n$ vertices percolates with a star doing so in one round, whilst any graph with less than $n-1$  edges will not percolate.

\subsection{Percolation thresholds} Whilst weak saturation looks at the {\em smallest} possible starting graphs that $H$-percolate, the statistical physics origins of bootstrap percolation motivates the investigation of {\em random} initial infection graphs. 
Balogh, Bollob\'as, and Morris~\cite{balogh2012graph} initiated the study of the threshold probability for when the random graph $G(n,p)$ becomes $H$-percolating.
 Here $G(n,p)$ denotes the Erd\H os-R\'enyi binomial random graph which is obtained by taking every edge on $n$ vertices independently with probability $p=p(n)$. We say an event holds asymptotically almost surely (a.a.s.\ for short) in $G(n,p)$ if the probability it holds tends to 1 as $n$ tends to infinity. As the property of being $H$-percolating is monotone increasing, 
 the classical result of Bollob\'as and Thomason \cite{bollobas1987threshold} gives that there is some threshold $p_c=p_c(n,H)$, such that if $p= \omega(p_c)$ then a.a.s.\ $G(n,p)$ is   $H$-percolating    whilst if $p=o(p_c)$ then a.a.s.\ $G(n,p)$ is not $H$-percolating. Considering Example \ref{ex:triangle}, we see that $p_c(n,K_3)$ is precisely the threshold for connectivity in $G(n,p)$, namely $\log n/n$ \cite{erdos1959random}.

   \subsection{Maximum running times} 
Another relevant question concerning percolation processes  is {\em how long} the virus could be spreading for. 
The study of the corresponding extremal parameter, concerning the worst case behaviour, was suggested by Bollob\'as (cf.\  \cite{bollobas2017maximum}).
We define \[M_H(n):=\max\{\tau_H(G): G \mbox{ is an } n\mbox{-vertex graph}\},\]
 to be the \textit{maximum running time} for the $H$-process on $n$ vertices. Returning to our Example \ref{ex:triangle}, we see that determining $M_{K_3}(n)$ is equivalent to maximising the distance between two vertices in the same connected component of  an $n$-vertex graph $G$. This is achieved when $G$ is a path, giving $M_{K_3}(n)=\lceil \log_2 (n-1) \rceil$.

 \subsection{The content and the goal of this survey} 
In all three settings outlined above, investigations focus on understanding how the infection rule $H$ influences the extremal functions of interest: the weak saturation number $\ws(n,H)$, the percolation threshold $p_c(n,H)$, or the maximum running time $M_H(n)$. 
 The purpose of this survey is to explore exactly this, presenting the current progress and highlighting the gaps in our knowledge. Indeed, as in all good mathematical problems, there are many mysteries and a full understanding of  $H$-processes presents  a formidable   challenge.

We will look at both weak saturation and percolation thresholds in this context and report the foundational work as well as  recent exciting breakthroughs and key unanswered questions. However, this survey is mostly concerned with maximum running times. In particular, our primary goal is to give an exposition of our recent systematic exploration \cite{FMSz1, FMSz2, FMSz3} of the parameter $M_H(n)$.   
We present  general proof methods and discover a rich landscape of theorems with several surprises therein. Along the way we build connections to other topics in  extremal graph theory  as well as probabilistic and additive combinatorics. 
We intend  this survey to act as an invitation to the topic, collecting and organising the state of the art and highlighting plenty of tantalising open problems and conjectures!

\subsection{Organisation and notation}

Maximum running times were first studied for cliques by Bollob\'as, Przykucki, Riordan, and Sahasrabudhe~\cite{bollobas2017maximum} and Matzke~\cite{matzke2015saturation}, and subsequently by Balogh, Kronenberg, Pokrovskiy, and the third author~\cite{balogh2019maximum}. It is these works that we treat in Section~\ref{sec:cliques and chains}, describing the development of insights that lay the foundation for our results and highlighting the motivations and connections to key concepts already present in works on weak saturation and percolation thresholds. 
In Section~\ref{sec:K5} we include 
 a novel concise  proof of a lower bound of Balogh et al.~\cite{balogh2019maximum} on the running time of the $K_5$-process.
This introduces central facts and definitions and explains the connection between our problem and additive combinatorics, a link that was first utilised in~\cite{balogh2019maximum}. 
In Section~\ref{sec:H perc} we touch on a question that merges  the perspective of maximum running time and percolating graphs. Namely, we prove that for a broad family of infection rules $H$,  the slowest  (up to a constant factor) initial infection configuration can also be chosen to be $H$-percolating. 
In Section~\ref{sec:quad} we start presenting our results, putting them in context and introducing key graph properties that contribute to extremely slow running times. 
 Sections~\ref{sec:sublin} and \ref{sec:graph parameters} investigate what property of an infection rule causes the process to stabilise quickly for every starting graph. Finally, Sections~\ref{sec:bipartite-rules} and \ref{sec:discon} are devoted to 
bipartite and disconnected infection rules, respectively.

Given two graphs $H$ and $H'$, the graph $H\cap H'$  denotes the graph on vertex set $V(H)\cap V(H')$ with edge set $E(H)\cap E(H')$. The graph $H\cup H'$ is the graph with vertex set $V(H)\cup V(H')$ and whose edge set is $E(H)\cup E(H')$.  The notation $H\sqcup H'$ is used to emphasise when the union is a \textit{disjoint union}, i.e.\ when $V(H)\cap V(H')=\emptyset$. 
Throughout the survey we will often identify graphs with their edge sets.  

 \subsection{What is not included} To close this introduction, we briefly mention some highly relevant directions that are omitted in order to keep this survey (relatively) concise.
 Indeed, there are vast bodies of work that provide context or study variants to the topics covered here. Luckily, there are already fantastic surveys available, by Currie, Faudree, Faudree and Schmitt \cite{currie2012survey} for (weak) saturation, and by Morris for bootstrap percolation \cite{morris2017bootstrap} and monotone cellular automata in general \cite{morrismonotone}. 

\vspace{1mm}

 {\bf{Neighbour bootstrap percolation.}} 
 Besides its intrinsic mathematical interest, graph bootstrap percolation \cite{balogh2012graph} was inspired by the classical notion of $r$-neighbourhood bootstrap percolation. There, one initially infects some \textit{vertices} of a specific (possibly infinite) graph $G$ and at each time step, a vertex becomes infected if at least $r$ of its neighbours are infected. This is a monotone cellular automaton (\`a la von Neumann~\cite{neumann1966theory} and Ulam~\cite{ulam1952random}). The process  was introduced by Chalupa, Leath and Reich~\cite{chalupa1979bootstrap} on the $m$-regular infinite tree, as a simplified model of ferromagnets. Subsequently
 an immense amount of literature developed around this concept in combinatorics and probability theory \cite{cancrini2008kinetically,holroyd2003sharp,janson2012bootstrap} as well as statistical physics \cite{adler2003bootstrap,adler_comparison_1989,aizenman1988metastability}. 
 Analogously to the graph bootstrap model, there is great interest in when the model \textit{percolates} with both random starting configurations \cite{balogh2009bootstrap,balogh2009majority,balogh2010bootstrap,biskup2009metastable,cerf2002threshold}  and extremal starting configurations \cite{balogh2006bootstrap,balogh2010bootstrap,dukes2023extremal,morrison_extremal_2018} being studied.  Running times of neighbour bootstrap percolation processes have recently been addressed from an extremal perspective by Przykucki and Shelton \cite{przykucki2019smallest},  Przykucki \cite{przykucki2012maximal} and Benevides and Przykucki \cite{benevides2013slowly,benevides2015maximum} and when the starting configuration is random by  Bollob\'as, Holmgren,  Smith, and Uzzell \cite{bollobas2014time},  Bollob\'as,   Smith, and Uzzell \cite{bollobas2015time} and Balister,  Bollob\'as, and  Smith \cite{balister_time_2016}.

\vspace{1mm}

 {\bf{Hypergraphs.}}
 The study of  weak saturation  with respect to hypergraphs has  been an important topic \cite{erdHos1991saturated,pikhurko2001uniform,pikhurko2001weakly,tuza1988extremal}. For the case of threshold probabilities, Noel and Morrison \cite{morrison2021sharp} considered a variant of the $H$-process  for hypergraphs $H$ where only certain copies of $H$ are activated (with some probability), as part of a more general theorem  generalising work of Kor\'andi, Peled and Sudakov \cite{korandi2016random}. For maximum running times, hypergraph cliques $H$ have also been considered  by Noel and Ranganathan  \cite{noel_running_2022}, Espuny D\'iaz, Janzer, Kronenberg and Lada \cite{espuny_diaz_long_2022} and Hartarsky and Lichev \cite{hartarsky_maximal_2022}.
  Whilst we do not dwell on the topic of hypergraphs for maximum running times  here, we  believe that it would be interesting to explore it further.

\section{Cliques and chains} \label{sec:cliques and chains}

We begin by exploring the $K_k$-process for all integers $k\geq 3$. 

\subsection{Weak saturation for cliques}  \label{subsec:weak-cliques}
Erd\H{o}s, Hajnal and Moon \cite{erdos1964problem} proved that for cliques $K_k$ with $k\geq 3$, one has $\sat(n,K_k)=(k-2)(n-k+2) + {k-2 \choose 2}=\binom{n}{2}-\binom{n-k+2}{2}$ and the unique graph that achieves this minimum is  
$K_{k-2}\vee \bar{K}_{n-k+2}$, the disjoint union of a $(k-2)$-clique and an independent set of order $(n-k+2)$ with a complete bipartite graph in between. 
 
For $k=3$, this matches what we observed in  the introduction, and $\ws(n,K_3)=\sat(n,K_3)=n-1$. Bollob\'as \cite{bollobas1968weakly} proved that the weak saturation number and saturation number coincide for cliques $K_k$ with $3\leq k \leq 7$ and later also conjectured \cite{bollobas1978extremal} that there is equality for all values of $k$. 
The conjecture was proved by Kalai who gave two proofs, one appealing to the geometric notion of rigidity  \cite{kalaiweaksat} and one using exterior algebras \cite{kalai1985hyperconnectivity}. An elegant proof inferring it from the Skew Set-Pairs Inequality appears in a paper of Alon~\cite{alon1985extremal}. The inequality can be derived from the pioneering work of Lov\'asz~\cite{lovasz_flats_1977}, which introduced the use of exterior algebra into combinatorics, but was first stated and proved explicitly by Frankl~\cite{frankl1982extremal} (see also \cite{kalai-intersection, alon1985extremal}).
Despite the keen interest and multiple proofs, the determination of $\ws(n,K_k)$ remains one of the intriguing combinatorial problems that to this date does not have a purely combinatorial solution.

One indication for why this is the case is the large number of non-isomorphic extremal examples. 
Indeed, unlike the saturation problem, where $K_{k-2}\vee \bar{K}_{n-k+2}$ is the unique graph achieving $\sat(n,K_k)=\ws(n,K_k)$, when considering $K_k$-percolating graphs, as observed already by Bollob\'as \cite{bollobas1968weakly}, one can build many more. Indeed, to obtain further constructions one can  arbitrarily add a new vertex of degree $k-2$ to a $K_k$-percolating  graph. 
The following provides a $K_4$-percolating example with the minimum number of edges, that looks very different to $K_{2}\vee \bar{K}_{n-2}$. 

 \begin{figure}[h]
    \centering
    \includegraphics[width=0.7\linewidth]{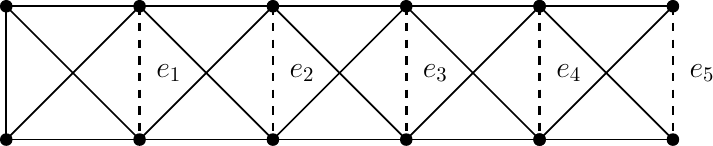}
    \caption{A simple $K_4$-chain.}
    \label{fig:k4chain}
    \end{figure}
\begin{Def}[A simple $K_k$-chain] \label{def:simple K4 chain}
Let $d\in \N$ and $n:=(k-2)d+2$. A \textit{simple $K_k$-chain} $(H_i,e_i)_{i=1}^d$ of \textit{length} $d$ is defined as follows. We take $H_1,\ldots,H_{d}$ to be a sequence of copies of $K_k$ such that for $1\leq i<j\leq d$ we have $V(H_i)\cap V(H_{j})$ empty unless $j=i+1$, in which case $V(H_i)\cap V(H_{i+1})$ has size two and  the edge at the intersection is  $e_i$. We let $e_{d}$ be an edge of $H_d$ disjoint from $e_{d-1}$ and we define $F_d$ to be the graph obtained by taking the union $\cup_{i=1}^dH_i$ after removing all the edges $e_i$.  
\end{Def}

 Note that for any even $n$, the graph $F_d$ with $d=n/2-1$  indeed provides an example of a $K_4$-percolating graph $G$ with $\ws(n,K_4)=2n-3$ edges.

\subsection{Percolation thresholds for cliques}
In their paper introducing the percolation threshold $p_c(n,H)$,  Balogh, Bollob\'as and Morris \cite{balogh2012graph} 
showed that for cliques of order $k\geq 3$ we have
\begin{equation} \label{eq:Kk threshold}  p_c(n,K_k)=n^{-1/{\lambda(K_k)}+o(1)},\end{equation}
where $\lambda(H):= \frac{e_H-2}{v_H-2}$ and the $o(1)$ in the exponent is hiding explicit poly-logarithmic factors in both the upper bound and lower bounds. For a sanity check we note that $\lambda (K_3)= 1$, as it should be to obtain the connectivity threshold in (\ref{eq:Kk threshold}).
In order to see where the exponent is coming from, consider the simple $K_k$-chains from Definition \ref{def:simple K4 chain}. Note that if an edge $e\in E(K_n)$ plays the role of $e_d$ for some copy of $F_d$ in $G(n,p)$ then certainly $e$ will be infected as one can see that the $K_k$-process on $F_d$ infects edge $e_i$ at time $i$ for $1\leq i \leq d$.  Now $n^{-1/\lambda(K_k)}$ is the threshold for when a fixed edge $e\in E(K_n)$ plays the role of a non-edge in a copy of $K_k^{--}$, the graph obtained from $K_k$ by removing a matching of size $2$. Indeed, the expected number of copies of $K_k^{--}$ that have $e$ playing the role of a non-edge is $n^{k-2}p^{e(K_k^{--})}$ which is constant at $p=n^{-1/\lambda(H)}$. Therefore, if $p$ is significantly above $n^{-1/\lambda(H)}$ (by some poly-logarithmic factors, say), then  \emph{any} $e\in E(K_n)$ plays the role of a non-edge in $K_k^{--}$, and we can string  together these copies to start to build some simple $K_k$-chain (or rather the resulting graph $F_d$) in reverse, that has $e$ as its final edge $e_d$. There are in fact many ways to do this and extending these to length $d\geq \log n$ say, we find one in which the final copy of $K_k^{--}$ added (which will be the first copy in the chain graph $F_d$) is in fact a copy of $K_k^-$ (that is, there is an extra edge) and so then we have a chain that causes the edge $e$ to be infected. With some more care one can apply the second moment method to  show that \textit{all} edges of $K_n$ play the role of $e_d$ in some copy of $F_d$ in $G(n,p)$ and will hence get infected.

There is a poly-logarithmic gap between the upper and lower bounds in \eqref{eq:Kk threshold}. For $K_4$ however, Balogh, Bollob\'as and Morris \cite{balogh2012graph} could determine $p_c(K_4)$ up to a constant. Subsequent work of Angel and Kolesnik \cite{angel2018sharp} and Kolesnik \cite{kolesnik2022sharp} established matching upper and lower bounds  respectively, giving a sharp threshold at $p_c(K_4)=1/\sqrt{3n\log n}$.  For larger cliques, Bartha and Kolesnik \cite{bartha2024weakly}  removed the poly-logarithmic factor from the upper bound in \eqref{eq:Kk threshold}. In a very recent breakthrough,  Bartha, Kolesnik, Kronenberg and Peled \cite{bartha2025sharp} announced the resolution of  a sharp threshold for $k\geq 5$ at $p_c(K_k)=(\gamma_k n)^{-1/\lambda(K_k)}$ for some explicit constants $\gamma_k$ that depend on \textit{Fuss-Catalan} numbers. Interestingly there is a change in behaviour for $k\geq 5$, evidenced by the lack of poly-logarithmic factors in the threshold.

\subsection{Maximum running times for cliques}\label{subsec:cliques} 

As we saw in the introduction, we have that $M_{K_3}(n)$ is logarithmic in $n$. 
The situation changes for $K_4$. Indeed, it is not difficult to see that the starting graph $F_d$ corresponding to a simple $K_4$-chain in Definition \ref{def:simple K4 chain} has running time $d=\frac{n}{2}-1$ and so $M_{K_4}(n)$ is at least linear in $n$. This construction is in fact a bit wasteful. 
Adding a new vertex to a percolating starting graph and connecting it to the endpoints of an edge that only gets infected in the last round, increases the running time by one. (This is in contrast with $F_d$, where the second endpoint of $e_i$ gets the same neighbours as the first, hence its addition does not increase the running time.)
Starting with $K_4$ minus an edge  and iterating leads to a construction with running time $n-3$. Bollob\'as, Przykucki, Riordan and Sahasrabudhe~\cite{bollobas2017maximum} and, independently, Matzke~\cite{matzke2015saturation}  showed that this  is best possible and  $M_{K_4}(n)=n-3$, for all $n\ge 3$. The upper bound here adopts the  idea of inductively ``growing'' a clique, showing that if the process does not stabilise by step $i$ then there  must be a clique of size $i+3$ in $G_i$.

For larger cliques $K_k$,  simple chains again give a linear lower bound but there is a key difference with the case of the $K_4$-chain. Indeed, for $k\geq 5$ the \textit{only} edges added in the $K_5$-process on $F_d$ are the edges $e_1,\ldots, e_d$ with $e_i$ added in the $i^{\text{th}}$ time step. This was not the case for $K_4$, where the process on $F_d$ percolates.  
The sparse final graph of the process on a simple $K_k$-chain gives hope that a long chain of copies of $K_k$ can somehow be made to wrap back onto itself whilst maintaining that each edge $e_i$ gets infected at time $i$.
Bollob\'as et al.\ \cite{bollobas2017maximum} and Matzke~\cite{matzke2015saturation} both managed to do this and achieve super-linear running times. 
This birthed the  concept 
of a ``chain construction'' which has since been the principal idea for deriving lower bounds for running times of $H$-processes. 
 We therefore introduce the relevant definitions in generality   allowing us to formalise and extend the ideas above.
Moreover, \textit{chains} as defined below also play a crucial role in deriving upper bounds (see Section \ref{sec:wheel}).

\begin{Def}[$H$-chains] \label{def:H-chain} Let $H$ be a connected graph, and let $\tau\geq 1$.
	An $H$\emph{-chain} of \emph{length} $\tau$ is a sequence $(H_i, e_i)_{i\in [\tau]}$ of  copies $H_i$ of $H$ together with edges $e_i\in E(H_i)$ such that 
	 $e_i \in E(H_i)\cap E(H_{i+1})$ for $i\in[\tau-1]$.
  	We call $G:=\cup_{i=1}^\tau H_i$ the \emph{underlying graph} of the $H$-chain and $G':=G-\{e_1,\ldots,e_\tau\}$ the \emph{starting graph} of the $H$-chain.
\end{Def}
Note that simple $K_k$-chains are certainly examples of $K_k$-chains  but Definition \ref{def:H-chain} is much more general as there is no restriction on how the copies of $H$ can intersect (other than the edges $e_i$ necessarily belonging to $H_i$ and $H_{i+1}$).

 Our aim  is to derive lower bounds on running times by showing that the edge $e_i$ is infected at time $i$ in the $K_k$-process on the starting graph of some $K_k$-chain. What extra conditions do we need for this?
For the edge $e_i$ to get infected at all, we do need some restriction on the intersections of edge sets of the graphs $H_i$.  
Also in order that edges $e_i$ are not infected too early in the $K_k$-process, we want that no other edge is infected.
Therefore we restrict the copies of $K_k^-$, the graph obtained by removing an edge from $K_k$. This leads to the following key lemma.

\begin{lemma}[A lower bound on running times for chains] \label{lem:chain lower}
Let $k\geq 5$ and $\tau\geq 1$ and suppose that $(H_i, e_i)_{i\in [\tau]}$ is a $K_k$-chain such that:
    \begin{enumerate}[\upshape \textbf{($\boldsymbol{\dagger}$)}]
    \item \label{(dag)}  For $1\leq i< j\leq \tau$, we have $E(H_i)\cap E(H_j)$  empty unless $j=i+1$, in which case the intersection is just $e_i$. 
    \end{enumerate} 
    \begin{enumerate}[\upshape \textbf{(*)}]
    \item \label{(*)} For any copy $F$ of $K_k^-$ in $G:=\cup_{i=1}^\tau H_i$, there is an $i\in [\tau]$ such that $F\subseteq H_i$.
    \end{enumerate} 
     Then if $G':=G-\{e_1,\ldots,e_\tau\}$ is  the starting graph of the chain, we have that  $\tau_H(G')\geq \tau$.
\end{lemma} 
\begin{proof} 
Setting $G_0=G'$ to be the starting graph, we show by induction that $G_i=G_{0}+\{e_1,\ldots,e_i\}=G_{i-1}+e_i$. The base case $i=0$ follows from the definition. For $i\geq 1$, supposing that the claim holds for previous $i$, we have that $H_i-e_{i}\subseteq G_{i-1}$ and so certainly $e_i$ is added to $G_i$. To see that no other edge is added, suppose  an edge $f\in G_i\setminus G_{i-1}$ completes a copy $F\subseteq G_{i-1}$ of $K_k^-$ to a copy of $K_k$ at time $i$. As $G_{i-1}\subseteq G$ by the induction hypothesis, property \ref{(*)} gives that  $F\subseteq H_j$ for some $j\in[\tau]$. As $f\notin G_{i-1}$, we have  $j\geq i$. If $j>i$, then $H_j\cap G_{i-1}=H_j-\{e_{j-1},e_j\}$ by the induction hypothesis and condition \ref{(dag)}.  
Therefore there are not enough edges in $H_j\cap G_{i-1}$ to contain $F$ and so we must have  $j=i$,  $F\subseteq H_i$ and $f=e_i$. 
\end{proof}

Simple $K_k$-chains, by construction, satisfy condition \ref{(dag)} for every $k\geq 3$. It is also easy to see that \ref{(*)} holds when $k\geq 5$.
\begin{obs} \label{obs:simple K5}
Suppose that $d\in \N$ and $(H_i,e_i)_{i\in [d]}$ is a simple $K_k$-chain, for some $k\geq 5$. If $F$ is some copy of $K_k^-$ in $G:=\cup_{i\in[d]}H_i$, then there is some $i\in [d]$ such that $F\subseteq H_i$. 
\end{obs}
Indeed, if $F$ was not contained in any $H_i$ then there would be a choice of $e_j$ such that removing the vertices of $e_j$ disconnects $F\cong K_k^-$, a contradiction. Note that this does \textit{not} hold in the case of $K_4$.

Chains satisfying the hypothesis of Lemma \ref{lem:chain lower} are what Bollob\'as, Przykucki, Riordan and Sahasrabudhe \cite{bollobas2017maximum} call \textit{good chains} and they used them to give lower bounds on $M_{K_k}(n)$ for $k\geq 5$. Matzke \cite{matzke2015saturation} also uses chains and gave an explicit construction of a good chain showing that $M_{K_k}(n)\geq \Omega(n^{3/2})$ for all $k\geq 5$.  Bollob\'as et al.~\cite{bollobas2017maximum} went  further, showing that 
\begin{equation} \label{eq:Kkfirstlower} M_{K_{k}}(n)\geq n^{2-\frac{1}{\lambda(K_k)}-o(1)},\end{equation}
with a random construction.  They build their chain by choosing one $H_i$ at a time, taking a uniformly random choice for the next $k-2$ vertices to complete  the vertices of $H_i$ and then taking $e_i\in H_i-e_{i-1}$ also as a random choice.  The exponent in \eqref{eq:Kkfirstlower} matches that in \eqref{eq:Kk threshold} and this is no coincidence. Indeed, the authors of \cite{bollobas2017maximum} show that in terms of unwanted (that is, not contained in some $H_i$) copies of graphs $F\subseteq K_k$ in $G=\cup_{i}H_i$, the graph $G$ resembles $G(n,p)$ of the same edge density. Therefore if a typical edge $e\in E(K_n)$ lies in some (unwanted) copy of $K_k^{--}$ in $G$, we can no longer use $e$ as it would create an unwanted copy of $K_k^-$. Thus the process for building the random chain lasts until a typical edge is no longer usable.

The exponents in \eqref{eq:Kkfirstlower} tend to 2 as $k$ tends to infinity, and as sketched above, are optimal with respect to a random construction of a good chain. 
Bollob\'as, Przykucki, Riordan and Sahasrabudhe \cite{bollobas2017maximum} contemplated that while it is ``tempting'' to think that their construction is also optimal in general, they had no real reason to believe this. Instead, motivated by the lack of non-trivial upper bounds in the problem, they opted for conjecturing only $M_{K_k}(n)=o(n^2)$, that is, that no $K_k$-process should be able to add a constant fraction of the edges of $K_n$ in distinct rounds.  
Surprisingly, even this more modest prediction turned out not to hold for $k\geq 6$. Balogh, Kronenberg, Pokrovskiy and  the third author~\cite{balogh2019maximum} gave an explicit construction of a starting graph with quadratic running time, hence proving that $M_{K_k}(n)=\Theta(n^2)$ for all $k\ge 6$. Theirs was also a chain construction but with a different flavour. In order to get longer running times, intuitively one needs the chains to intersect more significantly. One key observation of~\cite{balogh2019maximum} is that the condition \ref{(dag)} in Lemma \ref{lem:chain lower} can be  substantially weakened. Indeed, we do not really care if distinct $H_i$ and $H_j$ with $j>i+1$ intersect, as long as the edges $e_{j-1},e_j$ are not added before they should be. Thus  \ref{(dag)} can be replaced by 
 \begin{enumerate}[\upshape \textbf{($\boldsymbol{\dagger}'$)}]
    \item \label{(dag')}  For $1\leq i< j\leq \tau$, we have $e_j\notin E(H_i)$.  
    \end{enumerate} 
 One can check that the proof of Lemma \ref{lem:chain lower} goes through with this weaker condition.

 In~\cite{balogh2019maximum} this observation was exploited with the construction of a \textit{ladder chain} which we now sketch for the case of $H=K_6$. In fact, the construction we see  here is a slight variation of the one from~\cite{balogh2019maximum} and we instead follow  \cite{FMSz3}
 which generalised the construction to other $H$ (and introduced the terminology used here).  We will return to highlight the difference at the end of the following section  after seeing chain constructions in detail.

  The  idea  is to superimpose linearly many simple chains on the same vertex set. Each of these simple chains $(H^a_i,e_i^a)_{i\in [\tau_0]}$ will have some linear length, say  $\tau_0=\floor{n/8}$.  These can then be linked together via disjoint simple chains of constant length to form  one long continuous chain which will have quadratic length (we will see in the next section the details of such a linking procedure).  
  
  To construct the linearly long simple chains, we start by considering a bipartition of the vertices into two \textit{sides} which we think of as the \textit{left} side and a \textit{right} side. On each side we place a sequence of $2\tau_0$ copies of $K_3$, with consecutive copies intersecting in a singular vertex. Using these triangles, for  any \textit{slope} $a=0,1,\ldots,\tau_0$, we can consider a chain $(H^a_i,e^a_i)_{i\in [\tau_0]}$ which matches the  triangle in position $j$ on the left to the triangle in position $j+a$ on the right and places a complete bipartite graph between them (see Figure \ref{fig:ladderchain}). We will take an appropriate collection $A$ of slopes that will define our family of simple chains.

 \begin{figure}[h]
    \centering
    \includegraphics[trim=0 8.2cm 0 6.4cm,clip, width=\linewidth]{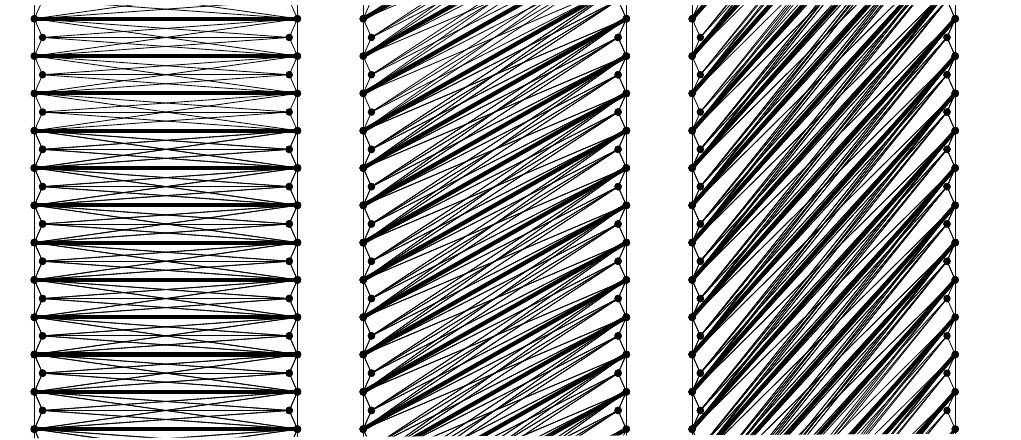}
    \caption{Sections of simple $K_6$-chains with slopes 0 (left), 4 (middle) and 8 (right). The edges $e_i$ of the simple chains are bold.}
    \label{fig:ladderchain}
    \end{figure}

Crucial to the  analysis of this chain and establishing the key condition \ref{(*)} of Lemma \ref{lem:chain lower} is the fact  that the graph  $G:=\cup_{a\in A,i\in [\tau]} H^a_i$  induces a sequence of triangles on both the left and right side. This in turn forces any copy $F$ of $K_6^-$ in $G$ to use three vertices on the left and three vertices on the right.  
Using this, one can show  that if  $F$ contains edges of the underlying graph $\cup_{i\in [\tau_0]}H^a_i$ for different slopes $a\in A$, then two of these must have difference at most three. We can avoid this by choosing only every fourth integer into $A$.

Interestingly, the construction  does not work for $K_5$. Indeed, the naive generalisation would place a sequence of triangles on the left side as before, and a path on the right side. For any pair of simple $K_5$-chains with slopes $a\neq a'$ using this bipartition, one can find copies of $K_5^-$ using one triangle on the left and edges corresponding to the different slopes $a,a'$ for the remaining vertices $v,v'$ of the copy. If we try to include these two simple chains as segments of a long chain, this will lead to unwanted copies of $K_5^-$, contradicting \ref{(*)}.  However with an alternative, more subtle approach the same authors \cite{balogh2019maximum} achieved almost quadratic time: \[M_{K_5}(n)\ge  n^{1-O(1/\sqrt{\log n})}=n^{2-o(1)}.\] 
In fact their construction gives that $M_{K_5}(n)\geq n\cdot r_3(n)$ where $r_3(n)$ is the size of a largest subset of $[n]$  free of three-term arithmetic progressions. The famous construction of  Behrend~\cite{behrend1946sets} gives that $r_3(n)\geq n^{1-O(1/\sqrt{\log n})}$, whilst it is well-known that $r_3(n)=o(n)$, as was originally shown by Roth \cite{roth_certain_1953} in 1953.  

We close this section with what is probably the most pertinent outstanding question in the pursuit of maximum running times, namely determining the asymptotics of $M_{K_5}(n)$.

\begin{conj}\label{conj:K5}
 [Bollob\'as-Przykucki-Riordan-Sahasrabudhe~\cite{bollobas2017maximum}] The maximum running time for infection rule $K_5$ is $M_{K_5}(n)=o(n^2)$. 
\end{conj}

\section{Proof highlight: $M_{K_5}(n)$ lower bound} \label{sec:K5}

   In this section, we prove a lower bound for $M_{K_5}(n)$ of the form $n^{2-o(1)}$. This was originally proved by Balogh et al.~\cite{balogh2019maximum},
    using a construction similar to the \textit{ladder chains} discussed in the previous section, although it uses a tripartition and slopes defined by Behrend's set free of arithmetic progressions. Our proof here uses \textit{dilation chains}~\cite{FMSz3} and is quite distinct, even though it also relies on  Behrend's construction.    
    In \cite{FMSz3}, we developed a theory of chain constructions (and in particular the dilation chains which we see here) to give general results that apply to a wide array of infection rules, see for example Section \ref{sec:chainsinsection5}. In this section, we  introduce  the key ideas of chain constructions and dilation chains in the context of $K_5$, where many technicalities can be avoided.

As with the ladder chains of the previous section, we will first superimpose linearly many simple chains on the same vertex set, each of linear length. Note that each of the chains $(H_i^a,e_i^a)_{i\in [\tau_0]}$ there can be obtained from the chain $(H^0_i,e^0_i)_{i\in [\tau_0]}$ with slope $0$ by leaving the left side fixed and translating the vertices on the right side by  $a$ triangles. The main difference for dilation chain constructions will be the way we choose the permutation that ``mixes up'' the vertices of a simple chain to obtain other simple chains in our family. The vertex set will be labelled by the elements of $\Z_p \setminus \{ 0\}$, which is equipped with nicely interacting additive and multiplicative operations. In the basic simple chain $(H^1_i,e^1_i)_{i\in [\tau_0]}$ the copies of $K_5$s will be laid out one after  the other, following the additive structure of $\Z_p \setminus \{ 0\}$. The family of simple chains $(H^a_i,e^a_i)_{i\in [\tau_0]}$ will be created using the multiplicative operation: the vertices of the basic simple chain are multiplied by a constant $a\in \Z_p \setminus \{ 0\}$. 
We  will again find that if we choose the set $A$ of dilation constants $a$ appropriately, then there will be no unwanted copy of $K_5^-$ when we superimpose the corresponding simple chains. Let us see this idea formally.

\subsection{The construction}  
Let $p\in\N$ be a prime and define $\tau_0:=\lfloor\frac{p-3}{3}\rfloor$. Let $A:=\{a_1,\ldots,a_q\}\subseteq \Z_p\setminus \{0\}$ be a subset of $q$ dilations. 

\vspace{1mm}

{{\bf Dilation chains.}} 
For an element $a \in \Z_p \setminus \{ 0\}$, the {\em dilation $K_5$-chain} $(H^a_i,f^a_i)_{i\in [\tau_0]}$ is defined on the vertex set $W=\{w_i : i \in \Z_p \setminus\{ 0\}\}$ such that 
\begin{equation}  V(H^a_i)=\{w_{a(3i-2)},w_{a(3i-1)},w_{a(3i)},w_{a(3i+1)},w_{a(3i+2)}\} \end{equation} 
and $f^a_i = w_{a(3i+1)}w_{a(3i+2)}$ for $i=[\tau_0]$. We also define $f^a_0=w_aw_{2a}$.

The chain $(H^1_i,f^1_i)_{i\in [\tau_0]}$ is clearly simple and the others are obtained from it by multiplying the indices of the vertices by $a$. 
Since $p$ is chosen to be a prime, the map $x\mapsto ax$ is a bijection of $\Z_p \setminus \{ 0 \}$ to itself so $(H^a_i,f^a_i)_{i\in [\tau_0]}$ is also a simple $K_5$-chain.

\vspace{1mm}
    
{\bf{Linking chains.}} For each element $a_j\in A$, the dilation chain $(H_i^{a_j},f_i^{a_j})_{i\in [\tau_0]}$ will feature as subchains of our final chain $(H_i,e_i)_{i\in[\tau]}$. For ease of notation we use $(H_i^j,f_i^j)_{i\in [\tau_0]}$ instead.
In order to link these simple chains up, we introduce short simple $K_5$-chains $(L_i^j,g_i^j)_{i\in [3]}$ for $j\in [q-1]$, each of length 3.
The $j^{\text{th}}$ linking chain will link the last $K_5$ of the $a_j$-dilation chain to the first $K_5$ of the $a_{j+1}$-dilation chain.
That is, two vertices of $L_1^j$ (disjoint from $g_1^j$) will be the vertices of the edge $f^j_{\tau_0}$ and $g_{3}^j$ will coincide with $f^{j+1}_0$, see Figure \ref{fig:linkchain}. All other 7 vertices of $U_j:=\cup_{i\in [3]}V(L_i^j)$ will be disjoint from $W$ and also disjoint from $U_{j'}\setminus (W\cap U_{j'})$ for  choices of $j'\neq j$. 

\begin{figure}[h]
    \centering
    \includegraphics[width=0.8\linewidth]{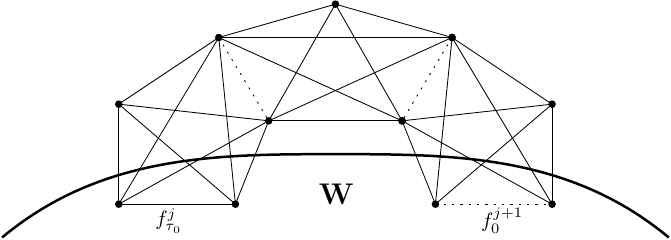}
    \caption{A linking $K_5$-chain $(L_i^j,g_i^j)_{i\in [3]}$ with the edges $g_i^j$ as dashed lines.}
    \label{fig:linkchain}
    \end{figure}

Note that in order to have that $(L_i^j,g_i^j)_{i\in [3]}$ is a well-defined simple $K_5$-chain, we need that $f^j_{\tau_0}\in E(L_1^j)$ and $g_{3}^j=f^{j+1}_0\in E(L_{3}^j)$ are vertex-disjoint in $W$. We will ensure this when we select the set $A$ of dilations. 

\vspace{1mm}

{\bf{Bringing everything together.}} Finally, we define our chain $(H_i,e_i)_{i\in [\tau]}$ which is simply obtained by concatenating the sequences, starting with $(H^1_i,f^1_i)_{i\in [\tau_0]}$ followed by $(L^1_i,g^1_i)_{i\in [3]}$ and then $(H^2_i,f^2_i)_{i\in [\tau_0]}$ and so on, alternating between dilation chains and linking chains, until we finish with the sequence $(H^q_i,f^q_i)_{i\in [\tau_0]}$. This  defines a $K_5$-chain which has length  
$\tau:=\tau_0q+3(q-1)$ and lies on $V=W\cup \left(\cup_{j\in [q-1]}U_j\right)$ which has size $|V|=p-1+7(q-1)$.

\subsection{The set of dilations.}
In this section, we collect properties of the set $A$ that we will need for our dilation chain construction above to work. 

First of all, as we promised above, we make sure that $f^j_{\tau_0}\cap f^{j+1}_0 = \emptyset$ for every $j\in [q-1]$, so the $j^{\text{th}}$ linking chain is simple. Furthermore, in order to use Lemma \ref{lem:chain lower} for the running time of $(H_i,e_i)_{i\in[\tau]}$, we need to establish properties \ref{(*)} and \ref{(dag)} for it. Both of these depend on certain restrictions about how the edges of different dilation chains interact with each other.
At the heart of the proof of all of these lies a general observation, relating the edges of the graph $G^j:=\cup_{i\in [\tau_0]}H^j_i$ of the $j^{\text{th}}$ dilation chain to the existence of a congruence modulo $p$.

\begin{obs} \label{obs:edges dil}
If $w_xw_y\in E(G^j)$ for $x,y\in [p-1]$ and $j\in [q]$, then there is some $\alpha\in \Z\setminus \{0\}$ with $|\alpha|\leq 4$  such that
\[x-y=\alpha a_j \mbox{ mod } p. \]
\end{obs}
\begin{proof}
As $w_xw_y\in E(G^j)$, there is some  $i\in[\tau_0]$ with $w_x,w_y\in V(H^j_i)= \{w_{a_j(3i-2)},\ldots,w_{a_j(3i+2)}\}$. 
Therefore, there are some distinct $\ell_x,\ell_y\in \{-2,-1,0,1,2\}$ such that $x=a_j(3i + \ell_x)$ and $y=a_j(3i +\ell_y)$ modulo $p$. Then  $x-y=\alpha a_j$ modulo $p$, where $\alpha= \ell_x-\ell_y\neq 0$. 
\end{proof}

The properties we need for our main proof forbid that edges from two or three different dilation chains form certain small subgraphs in $G$. The above observation translates these subgraphs to certain linear equations modulo $p$, in at most three variables, having a solution from $A$. 
This leads us to the next theorem, based on the famous construction of Behrend avoiding $3$-term arithmetic progressions. It gives us large subsets $A$ of $\Z_p$ which do not have non-trivial solutions to \textit{any} of the linear equations we will care about. 

\begin{theorem}[Behrend \cite{behrend1946sets}] \label{lem:beh stre}
    There exists $C'>0$ such that for any prime $p\in \N$ there is some set $A\subset \Z_p\setminus \{0\}$ such that $|A|\geq p^{1-C'/\sqrt{\log p}}$ and if \[\alpha_1a_1+\alpha_2a_2+\alpha_3a_3=0 \mbox{ mod } p\]
    for some $a_i\in A$ and $\alpha_i\in \Z$ with $|\alpha_i|\leq 4$    for $i\in [3]$ then $\sum_{i=1}^3\alpha_i=0$ and $a_i=a_j$ for all $i,j\in [3]$ with $\alpha_i,\alpha_j\neq 0$. 
\end{theorem}

Note that in particular, by fixing $\alpha_1=\alpha_2=1$ and $\alpha_3=-2$, the set $A$ from Theorem \ref{lem:beh stre} avoids three-term arithmetic progressions $x,x+d,x+2d$ with $x,d\in \Z_p$ and $d\neq 0$. The original proof of Behrend \cite{behrend1946sets} simply gave sets avoiding these arithmetic progressions but it is well-known \cite{ruzsa_solving_1993,shapira2006behrend} that his construction can be adapted to avoid non-trivial solutions to any constant number of linear equations in at most three variables. 

\vspace{1mm}

{\bf Distinct dilation chains are edge-disjoint.} 
To guarantee \ref{(dag)} we  need to make sure that $G^j$ and $G^{j'}$ do not share an edge if $j\neq j'$. 
If there was some edge $w_xw_y\in G_j\cap G_{j'}$ then by Observation~\ref{obs:edges dil} we have that there are some $\alpha, \alpha'\in \Z\setminus \{0\}$ with $|\alpha|,|\alpha'|\leq 4$ such that $\alpha a_j - \alpha' a_{j'} = 0$ modulo $p$. This contradicts our choice of $A$ from Theorem~\ref{lem:beh stre}, since $a_j, a_{j'} \in A$ and $a_j\neq a_{j'}$. 

\begin{cor} \label{cor:disj}
The graphs $G^j$ with $j\in [q]$ are pairwise edge-disjoint.
\end{cor}

{\bf Preventing ``unwanted'' copies of $K_5^-$.}
Condition \ref{(*)} forbids that a copy $F$ of $K_5^-$ in $G$ shares an edge with more than one  $H_i$. The most problematic part of showing this turns out to be  if $F$ shares an edge with more than one dilation chain.  
Here is where we use the full power of the properties of our set $A$ from Theorem~\ref{lem:beh stre}. 

\begin{lemma} \label{lem:K_5^-}
For any copy $F$ of $K_5^-$ in $\cup_{j\in [q]}G^j$, there is a $j\in [q]$ with $F\subseteq G^j$.  
\end{lemma}
The proof of this relies on
a simple combinatorial observation implying that in order to avoid copies of $K_5^-$ that have edge from different dilation chains, it is enough to avoid triangles with that property.

\begin{fact} \label{obs: K5 col}
Any non-monochromatic edge-colouring of $K_5^-$ contains a copy of  $K_3$, that is non-monochromatic.
\end{fact}

\begin{proof}[Proof of Lemma \ref{lem:K_5^-}]
Consider a colouring of $\cup_{j}G^j$ that colours an edge $j$ if it belongs to $G^j$. If $F$ is not a subgraph of any of the $G^j$ then it is not monochromatic and hence there is a triangle $F'$ which is also not.
Let $x,y,z\in \Z_p$ such that the vertices of $F'$ are $w_x,w_y,w_z\in W$ and let $j_1,j_2,j_3\in [q]$ such that $w_xw_y\in G^{j_1}$, $w_xw_z\in G^{j_2}$ and $w_yw_z\in G^{j_3}$.  By Observation \ref{obs:edges dil}, we have that for $i=1,2,3$ there are $\alpha_i\in \Z\setminus \{0\}$ such that $|\alpha_i|\leq 4$ and 
 \[x-y=\alpha_1 a_{j_1}, \quad z-x=\alpha_2a_{j_2} \quad \mbox{ and } \quad y-z=\alpha_3a_{j_3},\]
 modulo $p$.
 Summing these equations gives that $\sum_{i=1}^3\alpha_ia_i=0$ modulo $p$ and thus, as all the $\alpha_i$ are non-zero and $A$ satisfies the conclusion of Theorem \ref{lem:beh stre}, we  have that $a_{j_1}=a_{j_2}=a_{j_3}$, contradicting that $F'$ is non-monochromatic. 
 \end{proof}

{\bf Linking chains are simple.} Finally, we also rectify what we left open at the end of the definition of linking chains: to ensure that they are simple. Namely, we need to see that for every $j \in [q-1]$ the edge $f^j_{\tau_0}$ and the edge $f^{j+1}_0$ are disjoint. The vertices of $f^{j+1}_0$ are $w_{a_{j+1}}$ and $w_{2a_{j+1}}$ 
and the vertices of $f^j_{\tau_0}$ are among $w_{a_j(p-\alpha)}$, where $\alpha \in [4]$.
If two of these were to coincide then there was $\alpha' \in [2]$ and $\alpha \in [4]$
such that $\alpha' a_{j+1}- \alpha a_j = 0$ modulo $p$, which contradicts the property of our set $A$. 
\begin{cor} \label{obs: linking simple}
Every linking chain is simple.
\end{cor}

\subsection{Finishing the proof} 
Finally  we show that for appropriate choice of $p$ and $A$, the $K_5$-chain $(H_i,e_i)_{i\in [\tau]}$ satisfies the conditions of Lemma \ref{lem:chain lower}, hence the starting graph $G':=(\cup_{i\in [\tau]}H_i)-\{e_1,\ldots,e_\tau\}$ has running time $\tau =n^{2-o(1)}$. 
To this end, for $n\in \N$ sufficiently large, we let $p\in\N$ be a prime such that $n/4\leq p\leq n/2$, which exists by Bertrand's postulate/Chebyshev's theorem, and fix $\tau_0:=\lfloor \frac{p-3}{3}\rfloor$ as before. 
Let $A:=\{a_1,\ldots,a_q\}\subseteq \Z_p\setminus \{0\}$ with $q\geq p^{1-C'/\sqrt{\log p}}$ for some $C'>0$ be the set from Theorem  \ref{lem:beh stre}. Then  for an appropriate $C>0$ we have
\[\tau:=\tau_0q+3(q-1)\geq pq/4\geq  n^{2-C/\sqrt{\log n}}= n^{2-o(1)}.\]
 The chains all lie on $V=W\cup \left(\cup_{j\in [q-1]}U_j\right)$ which has size $|V|=p-1+7(q-1)\leq n$. 

We are left to show conditions \ref{(*)} and \ref{(dag)} of Lemma \ref{lem:chain lower}. For \ref{(dag)} note that by construction we do have that $E(H_i) \cap E(H_{i+1}) = e_i$ for every $i\in [\tau]$. For $2\leq i +1 < i'\leq \tau$ we must show that $E(H_i)\cap E(H_{i'})=\varnothing$. Corollary \ref{cor:disj} applies here unless $H_i$ and $H_{i'}$ belong to the same dilation chain or same linking chain. But these are simple (cf.\ Corollary~\ref{obs: linking simple}) and the property holds by construction.

To establish property \ref{(*)} we fix $F$ to be some copy of $K_5^-$ in $G:=\cup_{i\in [\tau]}H_i$ and show that there is some $i\in [\tau]$ such that $F\subseteq H_i$. Let us first assume that $F$ has a vertex $u$ outside of $W$, say $u\in U_j \setminus W$ for some $j\in [q-1]$. In case $V(F) \subseteq U_j$ then we are done by Observation \ref{obs:simple K5} as $G[U_j]$ only contains the edges of the $j^{\text{th}}$ linking chain which is simple by Corollary~\ref{obs: linking simple}.  
Otherwise  there is some $w \in V(F)\setminus U_j$, which we show is impossible. 
On the one hand $u$ and $w$ should have at least three common neighbours as they are the non-adjacent vertices of a copy of $K_5^-$ in $G$. This is because all neighbours of $u$ are in $U_j$ (and $w$ is not). On the other hand, we claim that their common neighbourhood is contained in one of the border edges $f_\tau^j$ and $f_0^{j+1}$ of the $j^{\text{th}}$ linking chain, providing a contradiction. 
Indeed, all $U_j$-neighbours of $w$ are in the union $f_{\tau_0}^j \cup f_0^{j+1}$, yet $u$, whose neighbourhood is contained in $U_j$, cannot have neighbours in both $f_{\tau_0}^j$  and $f_0^{j+1}$ (since these edges are disjoint by Corollary~\ref{obs: linking simple} and the linking chain has three copies of $K_5$).

Thus we have $V(F)\subseteq W$ and hence $F\subseteq \cup_{j\in [q]}G^j$. 
In this case Lemma~\ref{lem:K_5^-} implies that $F\subseteq G^j$ for some $j\in [q]$. As each dilation chain is simple by construction, Observation~\ref{obs:simple K5} can then be used to deduce that there is some $H_i$ containing $F$.  
$\Box$ 

\vspace{1mm}
The idea of the linking chains which we used for our lower bound on $M_{K_5}(n)$ originated in the proof of Balogh et al.~\cite{balogh2019maximum} for $K_5$. In our recent work 
 \cite{FMSz3}, we axiomatised the procedure and give general conditions for when chains can be linked to give longer chains. This reduces the problem of finding one long chain to finding many chain segments, which is often conceptually much easier. Indeed, for example for the ladder chains discussed in Section \ref{subsec:cliques}, by using linking chains one can give the simple chains with different slopes separately and does not need to worry about one leading into another. In fact, in their original proof showing that  $M_{K_k}(n)$ is quadratic for $k\geq 6$, Balogh et al.~\cite{balogh2019maximum} did not use linking chains. Instead, they used unequal sides of the partitions and wrapped the triangles around so that each time the chain visits a certain triangle, it does so with respect to a different slope.

\section{$H$-percolating graphs} \label{sec:H perc}
Having explored $K_k$-processes in detail, the rest of the survey will be concerned with the influence of the infection rule $H$ on the $H$-process, our main focus being on the running time $M_H(n)$. Before this, let us briefly take stock and recap what we have seen in the clique case.  In particular, it is interesting to consider the interplay between the running time of the bootstrap process and whether or not it percolates. In Section \ref{sec:cliques and chains}, we saw that $\ws(n,K_k)=\sat(n,K_k)$ and so there are edge-minimal $K_k$-percolating starting graphs $G$ that percolate in just one step. The set of extremal constructions for weak saturation is  larger though and  there are edge-minimal $K_k$-percolating graphs which are much slower to percolate, for example the simple $K_4$-chains of Definition \ref{def:simple K4 chain} which take linear time to percolate, asymptotically meeting the slowest possible running time  $M_{K_4}(n)$. For larger $k\geq 5$, as noted by Matzke \cite{matzke2015saturation},  the wide class of tight examples for $\ws(n,K_k)$ given by Bollob\'as \cite{bollobas1968weakly} using an inductive argument, all percolate in linear time and so fall far short of the maximum running time $M_{K_k}(n)$ which is quadratic (or at least almost quadratic for $K_5$). 

We also looked at the threshold $p_c(n,K_k)$ for when a random graph $G(n,p)$ is $K_k$-percolating. Above this threshold, we saw that typical edges of $K_n$ play the role of the final edge $e_d$ of some $F_d$ generated by a  $K_k$-chain after removing the edges $e_i$. Moreover, we can take that this chain is logarithmic in size and so random graphs above the threshold will also percolate fast. In fact, this was studied in detail by Gunderson, Koch and Przykucki \cite{gunderson2017time} who showed that just above the threshold $p_c(n,K_k)$ the random graph will a.a.s.\ $K_k$-percolate in $O(\log\log n)$ steps and this time decreases as $p$ increases. 

On the other hand, the chain constructions that we saw  to  lower bound  $M_{K_k}(n)$ for $k\geq 5$ are far from being $K_k$-percolating. Indeed, it was crucial to our analysis via Lemma \ref{lem:chain lower} that only the edges $e_i$ of the chain will be added during the $K_k$-process. This begs the question as to whether it is possible to have a starting graph $G$ such that the $K_k$-process percolates but does so slowly. The main topic of this section is a new result which indicates that it is possible for $H$-percolating  graphs to be slow with respect to the $H$-process, even asymptotically achieving the optimal maximal running time $M_{H}(n)$. 
Theorem \ref{thm:slow perc} below was noted, but not proven, by Bollob\'as, Przykucki, Riordan and Sahasrabudhe~\cite{bollobas2017maximum} for the case $H=K_k$. Here we show the result for the large class of \textit{$(2,1)$-inseparable} infection rules $H$.  

    \begin{Def}\label{def:inseparable}
    A graph $H$ with an edge is called {\em $(2,1)$-separable} if it can be disconnected by deleting at most two vertices (and all incident edges) and one further edge. Otherwise $H$ is called {\em $(2,1)$-inseparable}. 
    \end{Def}

Note that every $(2,1)$-inseparable graph is $3$-connected as it has at least five vertices and cannot be disconnected by deleting two vertices. Furthermore every $4$-connected graph is $(2,1)$-inseparable. Indeed, if the removal of an edge $e$ and two vertices $z_1,z_2$ disconnected a $4$-connected graph $H$ then so would the removal of $z_1, z_2$ and an appropriate endpoint of $e$ (as $H$ has at least five vertices).  
The notion of $(2,1)$-inseparability  plays a crucial role with respect to chain constructions. Indeed, it was precisely this property of $K_5$  that was used to prove  Observation \ref{obs:simple K5} showing that simple chains do not induce unwanted copies of $K_5^-$. 

\begin{theorem} \label{thm:slow perc}
For any $(2,1)$-inseparable $H$, there is some  $C=C(H)$ such that for all $n\in \N$, there is an $H$-percolating graph $\tilde{G}$  with $v(\tilde G)\leq C n$ and  $\tau_H(\tilde{G})\geq M_H(n)$.
\end{theorem}

Following \cite{bollobas2017maximum}, we define \[M'_H(n):=\max\{\tau_H(G): G \mbox{ is an } n\mbox{-vertex } H\mbox{-percolating graph}\}.\]
Then clearly one has that $M_H(n)\geq M'_H(n)$. Whilst Theorem \ref{thm:slow perc} does not quite give that $M_H(n)=O(M'_{H}(n))$ due to our lack of understanding (see Conjecture \ref{conj:lim}) of the functions $M_H(n)$ and $M'_{H}(n)$, it does give a strong indication that this is the case. 
 The $(2,1)$-inseparability is crucial in our proof of Theorem \ref{thm:slow perc} as we adopt ideas from chain constructions. As we will see later, there are graphs  that are $(2,1)$-separable for which we can still use chain constructions, for example the wheel graph or bipartite graphs which satisfy an analogue of inseparability \cite{FMSz3}. For these, it is likely that the proof ideas for Theorem \ref{thm:slow perc} can be extended. For other graphs $H$ for which we know $M_H(n)$, the constructions for the lower bounds are often $H$-percolating already. We therefore conjecture the following. 

\begin{conj} \label{conj:slow perc}
For all infection rules $H$, we have $M_H(n)=\Theta(M'_{H}(n))$.
\end{conj}

Before proving Theorem \ref{thm:slow perc}, we briefly survey what is known about  general infection rules $H$  in the context of weak saturation and percolation thresholds. 

\subsection{Minimal $H$-percolating graphs} \label{sec:wsat H}

Other than $H=K_k$, there are several results that determine $\ws(n,H)$ exactly. Restricting here to connected $H$, Borowiecki and Sidorowicz \cite{borowiecki2002weakly} noted that $\ws(n,C_k)=n-1$ if $k$ is odd and $n$ when $k$ is even. Indeed, any spanning tree will percolate if $k$ is odd and if $k$ is even an extra edge is needed, otherwise the final graph will be complete bipartite. For trees $T$, Faudree, Gould and Jacobson \cite{faudree2013weak} showed that  $\ws(n,T)$ does not depend on $n$ and for any $t$-vertex tree, one has $t-2\leq \ws(n,T)\leq \binom{t-1}{2}$ with both bounds being tight for certain trees.  Further weak saturation numbers for trees were determined by Pu and Cui \cite{cui2019weak} and recently by Chen, Liu and Yang \cite{chen2025note} who show that for any rational $\alpha\in [1,2]$ there is a family of $t$-vertex trees whose weak saturation numbers are of order $t^\alpha$.  Also recently, Kronenberg, Martins and Morrison \cite{kronenberg_weak_2021} determined $\ws(n,K_{k,k})$ and $\ws(n,K_{k,k+1})$ and complete bipartite graphs whose sizes grow with $n$ were considered in  \cite{akhmejanova2025weak}. 

Moving away from exact results to asymptotics, Alon \cite{alon1985extremal} proved that for every $H$, there is some constant $w_H\geq 0$ such that $\ws(n,H)=(w_H+o(1))n$. What controls the constant $w_H$? Faudree, Gould and Jacobson \cite{faudree2013weak} established that it is closely related to the minimum degree $\delta(H)$ of $H$, showing that 
\begin{equation} \label{eq:wsat interval}\frac{\delta(H)}{2}-\frac{1}{\delta(H)+1}\leq w_H\leq \delta(H)-1.\end{equation}
In fact, Terekhov and Zhukovskii \cite{terekhov2025weak} found an error in the lower bound which they then fixed, as well as showing that every multiple of $1/(\delta(H)+1)$  in the interval is realisable as the limit $w_H:=\lim \ws(H)/n$ for some graph $H$. Very recently Ascoli and He \cite{ascoli2025rational} went further, characterising all the rationals that are realisable as $w_H$ for some $H$. Interestingly, they showed that all possible rationals at least $3/2$ are realisable whereas below $3/2$ the realisable values are far more sparse with just one accumulation point at 1. They conjecture that all possible values of $w_H$ are rational.

We conclude this section by remarking that there are still many open questions regarding the behaviour of $\ws(n,H)$ and the associated constants $w_H$. In particular, it would be interesting to determine what properties of $H$ control the position of $w_H$ in the interval \eqref{eq:wsat interval}. What about for a typical $H$?

\begin{prob}
What is $w_H=\lim \frac{\ws(n,H)}{n}$ when $H=G(k,p)$?
\end{prob}

\subsection{$H$-percolating random graphs} \label{sec: H perc thresh}

With the same proof as for cliques $K_k$, Balogh, Bollob\'as, and Morris~\cite{balogh2012graph} showed an upper bound on $p_c(n,H)$ of the form $n^{-1/\lambda (H)+o(1)}$ for \textit{balanced} infection rules $H$, where we recall that $\lambda(H) := \frac{e_H - 2}{v_H -2}$  \eqref{eq:Kk threshold} and the $o(1)$-term hides logarithmic factors. Here {\em balanced} means that $\frac{e_F-1}{v_F-2} \leq \lambda (H)$ for every proper subgraph $F\subset H$ with $v_F \geq 3$. The balanced condition is necessary due to their use of the second moment method to show that typical edges lie at the end of simple $H$-chains.  The authors also show that all cycles $C_k$ act like $K_3$ and have $p_c(n,C_k)$ coinciding with the connectivity threshold $\log n/n$. Moreover, they show that the threshold for $K_{2,3}$ is also $p_c(n,K_{2,3})=\log n /n$ which gives an example $H$ where the exponent is not equal to $-1/\lambda(H)$. 
In fact, the paper also provides indication that the general picture is much more complex than just what $\lambda(H)$ offers. It is not difficult to see~\cite{balogh2012graph} that the existence threshold for a copy of $H$ minus an edge in $G(n,p)$, namely $\Omega(n^{-1/\lambda'})$ with $\lambda' =\lambda'(H) = \min_{e\in H} \max_{F\subseteq H-e} e_F/v_F$, provides a  lower bound on the $H$-percolation threshold for all infection rules $H$, and it was shown in \cite{balogh2012graph} 
that this lower bound is tight  when the infection rule $H$ has a pendant edge. They also provide further examples $H$  where the exponent of the threshold is $-\frac{1}{\lambda'(H)}$ and where it is strictly in between $-\frac{1}{\lambda'(H)}$ and $-\frac{1}{\lambda(H)}$. 

Bayraktar and Chakraborty~\cite{bayraktar2019k} establish that for the complete bipartite graph $K_{r,s}$ the critical exponent coincides with the upper bound $-1/\lambda(K_{r,s})$ of Balogh et al.~\cite{balogh2012graph} when $3\leq r\leq s \leq (s-2)^2 +s$, i.e. when $K_{r,s}$ is balanced. 

For general $H$ with $\delta(H) \geq 2$ and $v(H) \geq 4$,  Bartha and Kolesnik~\cite{bartha2024weakly} proved a lower bound $n^{-1/\lambda_*(H) +o(1)}$  which is governed by a more subtle function $\lambda_*(H) = \min \frac{e_H-e_F-1}{v_H -v_F}$, where the minimum is taken over all subgraphs $F\subset H$ with $2\leq v_F < v_H$. They show that for balanced graphs $\lambda^*= \lambda$, so together with the upper bound of \cite{balogh2012graph} the asymptotics of the exponent of the threshold is determined for all balanced graphs. 
Recently Bartha, Kronenberg, and Kolesnik~\cite{bartha2024h} proved $H=G(k, 1/2)$ is balanced a.a.s.\ and hence the critical exponent for the $H$-process percolating is equal to $-1/\lambda(H)$ for ``most'' infection rules $H$. 

For non-balanced $H$, one has that $\lambda_*$ is strictly smaller than $\lambda$ and may even be smaller than $\lambda'$. 
Therefore, for non-balanced infection rules $H$, the problem of determining the percolation threshold $p_c(n,H)$   is still an outstanding mystery. The graphs $K_{2,t}$ seem like an interesting example that is yet to be understood. Indeed, Bidgoli, Mohammadian and Tayfeh-Rezaie \cite{bidgoli_k_2t-bootstrap_2021} offer both  upper bounds and lower bounds  on the exponent for $p_c(n,K_{2,t})$. For $K_{2,3}$, their upper bound exponent had already been observed to be tight in \cite{balogh2012graph}. In~\cite{bidgoli_k_2t-bootstrap_2021} the tightness of their upper bound is also established for $K_{2,4}$. There the answer is $-10/13$, which is interestingly different from $-1/\lambda,-1/\lambda'$ and $-1/\lambda_*$.  

\subsection{Slow $H$-percolating graphs} \label{sec:slow perc}
In this section, we prove Theorem \ref{thm:slow perc} giving $H$-percolating graphs that percolate slowly. This proof has not appeared elsewhere and we believe it is of considerable interest to the community. 
Nonetheless, it is somewhat delicate and a bit of a detour from the main theme of the survey. Therefore the reader may want to skip ahead to Section \ref{sec:quad} on first read.

In order to prove Theorem \ref{thm:slow perc}, we follow a natural strategy. Given $n\in \N$, fix a starting graph $G$ that achieves running time $\tau:=M_H(n)$ and an edge $e_\tau$ that is added at time $\tau$ in the process on $G$, the aim is to create a new starting graph $\tilde{G}$ on $O(n)$ vertices that contains $G$ and does not interfere with the $H$-process on $G$ for the first $\tau$ steps of the process. The addition of  $e_\tau$ should then  trigger the rest of the graph to percolate. In particular, we want that for every edge $f$ on $V(G)$ that does \textit{not} lie in $\final{G}_H$, there is some gadget that is triggered by $e_\tau$ and forces $f$ to be infected. A natural candidate for such a gadget is a \textit{simple $H$-chain}. Here, as with our simple $K_k$-chains, a simple $H$-chain is just a sequence $(H_i,e_i)_{i\in \ell}$ of copies of $H$ with consecutive copies $H_i$ and $H_{i+1}$ intersecting in a single edge $e_i\in H_i\cap H_{i+1}$ (and no other intersections between the $H_i$). As we saw for $H=K_5$ in Observation \ref{obs:simple K5}, we have the following.
\begin{obs} \label{obs:insep simple} If $H$ is $(2,1)$-inseparable and $(H_i,e_i)_{i\in \ell}$ is a simple $H$-chain then for any $e\in E(H)$ and any copy $F$ of $H-e$ in $\cup_{i\in [\ell]}H_i$, there is some $i\in [\ell]$ such that $F\subseteq H_i$. 
\end{obs}

As a consequence of this observation, note that if $e_0\in H_1-e_1$ the graph $G'=\cup_{i=1}^\ell H_i-\{e_0,\ldots,e_\ell\}$ is \textit{$H$-stable}, that is, no edge is added in the $H$-process on $G'$. Indeed, for each $i\in [\ell]$, we have that $V(H_i)$ hosts $e(H)-2$ edges in $G'$ and so we cannot embed some copy of $H-e$ into $H_i$. 

Now armed with these simple $H$-chains, in a similar fashion to the linking chains used to lower bound $M_{K_5}(n)$, for each $f\in \binom{V(G)}{2}-\final{G}_H$, we can place a short simple chain $(H_i,e_i)_{i\in [\ell]}$ such that $e_\tau$ is some edge of $H_1-e_1$, $f=e_\ell$ and all other vertices of the chain are disjoint from $V(G)$, with the new vertices introduced by these chains being distinct for different choices of $f$. One immediate problem with this strategy is that we might be adding too many vertices. Indeed, it could be that there are quadratically many edges $f\in \binom{V(G)}{2}-\final{G}_H$ and each one requires adding constantly many new vertices. For this, there is an easy solution. Indeed, we do not have to infect \textit{all} of the edges $f$, just enough of them so as to trigger all edges to be added eventually. Indeed, as we saw in Section \ref{sec:wsat H}, there is some graph $G'$ on $n$ vertices with $(w_H+o(1))n\leq v(H)n$ edges that is $H$-percolating. Placing an auxiliary copy of $G'$ on $V(G)$ and infecting all edges $f$ which lie in $G'$, will be enough to trigger the process to infect all  $\binom{V(G)}{2}$ edges. 

There is a second problem which is more serious. Indeed, the above sketch shows how to infect all of the edges on $V(G)$ but in doing so, we added new vertices for our small chains $(H_i,e_i)_{i\in [\ell]}$ used to infect the edges $f$. There is no guarantee that the edges incident to these new vertices will be infected by this process. Indeed, the small chains were chosen to be vertex-disjoint from $V(G)$ \textit{precisely because} we do not want the process on the small chains $(H_i,e_i)_{i\in [\ell]}$ and the  process on $G$ to interact. One way to guarantee that vertices gain many neighbours in the $H$-process is if they have $\delta(H)-1$ neighbours in a set that will eventually become a large clique, as $V(G)$ will become after we ensure that all edges on $V(G)$ become infected. 

\begin{lemma}\label{lem:cliques}
    Let $G$ be a graph, and  $U, W \subseteq V(G)$.
    If $W$ is a clique of size at least $v(H)-1$ in $\final{G}_H$, and every vertex in $U$ has  at least $\delta(H)-1$ $G$-neighbours in $W$, then $U\cup W$ is also a clique in $\final{G}_H$.
    \end{lemma}

    \begin{proof}
    Let $u\in U$, and let $w_1,\ldots,w_{\delta(H)-1}\in W$ be pairwise distinct $G$-neighbours of $u$.
    For $w\in W\setminus\{w_1,\ldots,w_{\delta(H)-1}\}$, any map $\varphi: V(H) \to V(G)$ that sends a vertex of minimum degree to $u$, its neighbourhood to  $\{w,w_1,\ldots,w_{\delta(H)-1}\}$, and the remaining vertices of $H$ to the rest of $W$ is an embedding of $H$ minus an edge into $\final{G}_H$.
    Thus $uw$ completes a copy of $H$ and there must be a complete bipartite graph between $U$ and $W$ in $\final{G}_H$. All edges in $\binom{U}{2}$ then complete some copy of $H$ with other vertices chosen from $W$ and so indeed $U\cup W$ is a clique in $\final{G}_H$.
    \end{proof}

We will use this to trigger new vertices to join a clique with $V(G)$.  However, we cannot add $\delta(H)-1$ edges from every new vertex to $V(G)$ without interfering with the $H$-process on $G$. We will be able to do this for some subset $U$ of the new vertices though. This leads us to consider a new gadget which is similar to a  simple $H$-chain, but has the added property that if we trigger the infection of a clique on some vertices $U$ in the middle of the chain then the infection will cause the whole gadget to percolate. The following lemma details what we need from this gadget.

    \begin{lemma}[The gadget graph]\label{lem:gadget}
    For any $(2,1)$-inseparable $H$ on $k$ vertices, there exists a graph $\Gamma$ on $\ell\leq k^{4k}$ vertices together with an edge $e\in E(\Gamma)$ and a non-edge $f\in \binom{V(\Gamma)}2 \setminus E(\Gamma)$ such that
        \begin{enumerate}
        \item $\Gamma-e$ is $H$-stable (no edge is added in the $H$-process on $\Gamma-e$) and $f\in E(\final{\Gamma}_H)$,
        \item $\mathrm{dist}_\Gamma(e,f) \geq k$,
        \item there is  $U\subseteq V(\Gamma)$ such that $\Gamma \cup \binom{U}2$ is $H$-percolating and $\mathrm{dist}_\Gamma(U,e\cup f) \geq k$.
        \end{enumerate}
    \end{lemma}

    \begin{proof}
    Let $\ell := \ceil{\frac{2k^{3k+3}-2}{k-2}}$, and let $(H_i,e_i)_{i\in[\ell]}$ be a simple $H$-chain of length $\ell$.
    Label the vertices of the chain by $v_1,\ldots,v_r$, where $r:= 2+(k-2)\cdot\ell \geq 2k^{3k+3}$, such that the $i$\textsuperscript{th} copy of $H$ has vertex set $\Set{ v_j : j\in [(i-1)\cdot(k-2)+1,i\cdot (k-2)+2]}$.
    Write $\delta := \delta(H)$.
    Define $\Gamma$ by
    \[
        V(\Gamma) = \Set{v_1,\ldots,v_r} \quad \mbox{ and } \quad 
        E(\Gamma) = \cup_{i=1}^\ell E(H_i) \setminus \Set{e_1,\ldots,e_\ell} \cup E' \cup E'',
       \]
    where
        \begin{align*}
        E' &= \Set*{ v_iv_j : i\in [k^{3k}], j \in k^2\cdot[(i-1)(\delta-1)+1,i(\delta-1)]} \\
        E'' & =\Set*{ v_{r+1-i}v_{r+1-j} : i\in [k^{3k}], j \in k^2\cdot[(i-1)(\delta-1)+1,i(\delta-1)]},
        \end{align*}
    and let $e := v_1v_2$, $f:= e_\ell$, and $U := \Set{v_i : i\in[k^{3k}+1,r-k^{3k}]}$.

    To show (1), note that $e_i$ appears by step $i$ in the  $H$-process on $\cup_{i=1}^\ell H_i - \Set{e_1,\ldots,e_\ell}$ and in particular $f$ appears by step $\ell$.
    Since $\cup_{i=1}^\ell H_i - \Set{e_1,\ldots,e_\ell} \subseteq \Gamma$ we obtain that $f\in E(\final{\Gamma}_H)$.
    To see that $\Gamma-e$ is $H$-stable we show that there are no copies of $H$ minus an edge in $\Gamma-e$.
    Indeed, if $F$ is such a copy then $F$ is not contained in $\cup_{i=1}^\ell H_i - \Set{e,e_1,\ldots,e_\ell}$ because $(H_i,e_i)_{i\in[\ell]}$ is a simple $H$-chain (as noted after Observation \ref{obs:insep simple}).
    For this reason at least one edge of $F$ lies in $E'\cup E''$.
    Without loss of generality we can assume that $E(F)\cap E' \neq \varnothing$ and $E(F)\cap E''=\varnothing$ (using here that the largest vertices in an edge of $E'$ are of  distance greater than $k$ in $\Gamma$ from the smallest vertices in an edge of $E''$).
    Pick the largest $j\in [r]$ such that $v_j$ is an endpoint of an edge $e'$ in $E(F)\cap E'$. 
   Note that   $N_{F-e'}^{k-1}(v_j)\subseteq \{v_i : |i-j| \leq (k-1)^2  \}$ where $N_{F-e'}^{k-1}(v_j)$ denotes  all the vertices that can be  reached from $v_j$ by a path   in $F-e'\subseteq \Gamma-e'$ with at most $k-1$ edges. We used here that any edge $e''\in E'$ with endpoints $v_{i''},v_{j''}$ and $i''<j''<j$ must have $j''\leq j-k^2$. 
    Hence there can be no path between $v_j$ and the other endpoint of $e'$ in $F-e'$, which contradicts the assumption that $H$ is $(2,1)$-inseparable.

    As to (3), by construction every vertex $v_i$ with $i\leq k^{3k}$ has at least $\delta-1$ neighbours in $\{ v_{i+1},\ldots,v_{\floor{r/2}} \}$,
    and for $i\geq r-k^{3k}+1$, $v_i$ is adjacent to at least $\delta-1$ vertices in $\{ v_{\ceil{r/2}},\ldots,v_{i-1} \}$.
    By repeated applications of Lemma \ref{lem:cliques}, we see that $\Gamma \cup \binom{U}2$ indeed $H$-percolates.
    Let $s := \dist_{\Gamma}(e,U)$ and choose a shortest path $v_{i_1}\ldots v_{i_s}$ path from $e$ to $U$.
    For any $i\in [k^{3k}]$ and $j\in [r]$ with $v_iv_j\in E(\Gamma)$, we have $j \leq i(\delta-1)k^2$.
    Therefore,
        \begin{equation*}
        k^{3k} < i_s = i_1\cdot\frac{i_2}{i_1}\cdot\ldots\cdot \frac{i_s}{i_{s-1}} \leq 2 \cdot (\delta-1)^{s-1}k^{2(s-1)} < k^{3s}
        \end{equation*}
    and hence $s \geq k$.
    Similarly, $\dist_\Gamma(f, U) \geq k$.

    It remains to prove (2).
    As there are no edges between $\{ v_1,\ldots,v_{k^{3k}} \}$ and $\{ v_{r-k^{3k}+1},\ldots,v_r \}$, any path from an endpoint of $e$ to an endpoint of $f$ must use a vertex from $U$.
    This implies $\dist_\Gamma(e,f) \geq \dist_{\Gamma}(e, U) \geq k$.
    \end{proof}

    This puts us in good stead but we still need to find a way to attach the set $U$ for each new gadget to $\delta(H)-1$ vertices in $V(G)$ without interfering with the process on $G$. To do this, we simply add an independent set $I$ to $G$ and consider the auxiliary $H$-percolating graph $G'$ with linearly many edges as a graph on vertex set $V(G)\cup I$. Therefore our gadgets will force that $V(G)\cup I$ hosts a complete graph and we can attach the vertex sets $U$ to vertices in $I$ to trigger the percolation on the individual gadgets. One more idea is needed to guarantee that the edges between $U$ and $I$ for different gadgets do not interact with each other to give unwanted copies of $H$ minus an edge. We will use a bipartite graph of large girth to place these edges between gadgets and  $I$, recalling that the \textit{girth} of a graph is the length of the shortest cycle in the graph.  Specifically, we need the following.

     \begin{lemma}\label{lem:auxiliary_bipartite_graph}
    Let $k, d \in \N$ be fixed.
    Provided that $n\in\N$ is sufficiently large, there exists a bipartite graph $B$ with partite sets $X$, $Y$ of sizes $|X| = n$, $|Y| = 2kn$ such that every vertex in $Y$ has degree $d$ and the girth of $B$ is at least $k+1$.
    \end{lemma}

Bipartite graphs with many edges and high girth have been studied thoroughly in the literature, see for example \cite{lazebnik1995new,lazebnik_connectivity_2004,furedi2013history}. In fact, we will use them again in our construction of \textit{line chains} in Section \ref{sec:connect}.
We note that for our purposes we do not require the full strength of those results because the parameters $k$ and $d$ are fixed constants. Indeed, Lemma \ref{lem:auxiliary_bipartite_graph} can be obtained easily by considering a random bipartite graph with parts $X$ of size $n$ and $Y'$ of size $(2k+1)n$, with each edge present independently with probability $p=n^{-1+1/(2k)}$. By an application of Chernoff's theorem, a.a.s.\ all vertices in $Y'$ will have degree at least $d$ and by Markov's inequality, a.a.s.\ there will be at most $n$ cycles of length at most $k$. Therefore a.a.s.\ both these events happen and we can obtain the graph $B$ in Lemma \ref{lem:auxiliary_bipartite_graph} by first deleting from $Y'$ any vertex that participates in a cycle of length at most $k$ and then deleting further vertices of $Y'$ and edges to get $Y$ of size exactly $kn$ with each vertex incident to $d$ edges.  Finally then, we give the details of the proof of Theorem \ref{thm:slow perc}
	
    \begin{proof}[Proof of Theorem \ref{thm:slow perc}]
    Let $H$ be a $(2,1)$-inseparable $k$-vertex graph and  $G$  an $n$-vertex graph with running time $\tau_H(G) = M_H(n)$.
    We  construct a super-graph $\tilde{G} \supseteq G$  such that $\tau_{H}(\tilde G) \geq \tau_{H}(G)$ and $\tilde G$ is $H$-percolating.
    Introduce an independent set $I$ of $n$ new vertices, and let $G'$ be an auxiliary graph with vertex set $V(G)\cup I$ such that $|E(G')|\leq 2kn$ and $\final{G'}_H = K_{2n}$.
    Let $\Gamma$ be the graph from Lemma \ref{lem:gadget} with $\ell:=v(\Gamma)$ vertices, and for each $f\in G^*:=G'\setminus \final{G}_H$ introduce a set $W^f$ of $\ell-4$ new vertices.
    Fix an edge $e_\tau \in G_\tau\setminus G_{\tau-1}$ and place a copy $\Gamma^f$ of $\Gamma$ on $W^f\cup e_\tau \cup f$  such that for some subset $U^f\subseteq W^f$, (1)\textendash (3) of Lemma \ref{lem:gadget} hold for $e_\tau, f, U^f$ (with $e_\tau$ playing the role of $e$).

    Consider  an auxiliary bipartite graph $B$ between $I$ and $\Set{ U^f : f\in G^*}$ such that
        \begin{enumerate}[(i)]
        \item $\deg_{B}(U^f) = d:=|U^f|\cdot (\delta(H)-1)$ for all $f\in G^*$,
        \item $B$ has girth at least $k+1$.
	\end{enumerate}
    Such an auxiliary graph exists by Lemma \ref{lem:auxiliary_bipartite_graph} whenever $n$ is sufficiently large, deleting arbitrary vertices from the set $Y$ to get a set of size $e(G^*)$, if necessary. 
    For each $f\in G^*$ choose a partition $N_{B}(U^f) = \cup_{u\in U^f}N_u$ into $|U^f|$ sets of size $\delta(H)-1$.
    Define $\tilde G$ by
        \begin{align*}
        V(\tilde G) &= V(G) \cup I \cup \left(\cup_{f\in G^*} W^f\right), \\
        E(\tilde G) &= E(G) \cup \left(\cup_{f\in G^*} \left( E(\Gamma^f-e_\tau) \cup \{ vu: u \in U^f, v \in N_u \} \right)\right).
        \end{align*}
    With the above choices there is an edge in the auxiliary graph $B$  between $v\in I$ and $U^f$ whenever $v$ has a $\tilde G$-neighbour in $U^f$. Note also that $\tilde{G}$ has $v(\tilde{G})\leq 2n+ 2kn(\ell-4)\leq 4k^{4k}n$ vertices, as required. 

    First we show that $\tilde G$ percolates.
    We have $G_\tau \subseteq \tilde G_\tau$ since $G\subseteq \tilde G$.
    For each $f\in G^*$, $\Gamma^f\subseteq \tilde G_\tau$ and hence $\final{\Gamma^f}_H\subseteq\final{\tilde G}_H$.
    In particular, $f\in E(\final{\tilde G}_H)$.
    Therefore, $G'\subseteq \final{\tilde G}_H$.
    Recall that $G'$ is $H$-percolating and so $V(G')=I\cup V(G)$ is a clique in $\final{\tilde G}_H$.
    Every vertex in $\cup_{f\in G^*}U^f$ has $\delta(H)-1$ $\tilde G$-neighbours in $V(G')$, so $V(G')\cup\left(\cup_{f\in G^*}U^f\right)$ is a clique in $\final{\tilde G}_H$ by Lemma \ref{lem:cliques}.
    Since $\Gamma^f\cup\binom{U^f}2$ percolates, $V(\Gamma^f)$ is a clique in $\final{\tilde G}_H$.
    Each of the sets $V(\Gamma^f)$, $f\in G^*$, intersects $V(G')\cup\left(\cup_{f\in G^*}U^f\right)$ in at least $\delta(H)-1$ vertices.
    Therefore, repeated applications of Lemma \ref{lem:cliques} guarantee that $\final{\tilde G}_H$ is a complete graph.

    Finally, we prove that $\tau_H(\tilde G) \geq \tau$. We  show by induction that for all $0\leq t\leq \tau$ \[E(\tilde{G}_t)= E(G_{t})\cup (E(\tilde{G})\setminus E(G)).\] For $t=0$, this  simply holds by definition. For larger $t$, suppose  that we have proved the induction hypothesis up to step $t-1$. This means that the only edges that have been added  are the edges in the $H$-process on $G$ up to time $t-1$. In particular, the neighbourhood of any vertex $v\in I$ in $\tilde{G}_{t-1}$ is the same as its neighbourhood in $\tilde{G}$.
 
    Let $F$ be a copy of $H$ minus an edge in $\tilde G_{t-1}$.
    Suppose first that $V(F)$ intersects $I$. As $N_{\tilde G_{t-1}}(v)= N_{\tilde{G}}(v)\subseteq \cup_{f\in G^*}U^f$, the
     connectedness of $F$ and the inequality $\dist_{\Gamma^f}(e_\tau\cup f, U^f) \geq k$  imply that $V(F) \subseteq I\cup\left(\cup_{f\in G^*}W^f\right)$.
    Every $v\in V(F)\cap I$ has at most one $F$-neighbour in $U^f$ for each $f\in G^*$, and hence $|\{ f\in G^*: U^f\cap N_{F}(v) \neq \varnothing \}| \geq \delta(F) > 2$, using that $H$ is $(2,1)$-inseparable in the last inequality here. Moreover, for every $f\in G^*$ such that $U^f\cap V(F)\neq \varnothing$, we must have that there are at least two $B$-neighbours of $U^f$ in $I$ as otherwise we could disconnect $F$ by removing one vertex, contradicting the $(2,1)$-inseparability of $H$. 
    Therefore we have that    $B[(I\cap V(F))\cup\{U^f:U^f\cap V(F)\neq \varnothing\} ]$ has minimum degree greater than two and therefore contains a cycle. Each vertex in this cycle corresponds to (at least one) vertex of $V(F)$ and these vertices are distinct. Hence the cycle has length at most $v(F)=k$ which contradicts the assumption that the girth of $B$ is at least $k+1$.

    We have shown that $I\cap V(F) = \varnothing$, or equivalently, $V(F) \subseteq V(G) \cup \left(\cup_{f\in G^*}W^f\right)$.
    If $V(F)\cap W^f\neq \varnothing$ for some $f\in G^*$ and $V(F)\not\subseteq V(\Gamma^f)$ then either $f\cap V(F)$ or $e_\tau\cap V(F)$ is a vertex-cut of $H'$ because $\dist_{\Gamma^f}(e_\tau,f) \geq k$.
    This, however, would contradict the assumption that $H$ is $(2,1)$-inseparable.
    Consequently, $V(F)\subseteq V(G)$ and $F\subseteq G_{t-1}$ by the induction hypothesis. Therefore any edge that is added during the step $t$ of the $F$-process on $\tilde G$ belongs to $G_t\setminus G_{t-1}$ and we are done.
    \end{proof}

We remark that the constant $C=C(H)$ appearing in Theorem \ref{thm:slow perc} is probably unnecessary and a similar result should be achievable with a new starting graph $\tilde{G}$ on $(1+o(1))n$ vertices. An approach to do this by packing gadgets onto $o(n)$ vertices using a high girth (approximate) design was suggested by the anonymous referee. We did not pursue this direction, but we believe that it should indeed be possible. 

\section{(Almost) quadratic running times} \label{sec:quad}

In this section, we begin our exploration of $M_H(n)$ for general infection rules $H$. Our first point of interest will be to what extent the behaviour of large cliques (of size at least five) generalises to other infection rules $H$. The results of this section are from our paper \cite{FMSz3} and are achieved through a program of building a general framework for chain constructions,  in particular exploiting the ideas behind the ladder chains and dilation chains that we already encountered.  This method turns out to be  flexible and we are able to show extremely slow, that is (almost) quadratic, running times for a large array of infection rules $H$,  including many that are far from resembling the large cliques that we set out to generalise. 

For one of our almost quadratic examples, namely the wheel graph, we are also able to complement our lower bound with a non-trivial sub-quadratic upper bound. This establishes a new behaviour for maximum running times, exactly the type Conjecture~\ref{conj:K5} predicts for $K_5$.

\subsection{Minimum degree}

In our first generalisation of the results of Balogh et al.~\cite{balogh2019maximum} for cliques, we establish Dirac-type conditions for quadratic and almost quadratic running times.

    \begin{theorem}[\cite{FMSz3}] \label{thm;dense quad}
    Suppose $5\leq k\in \N$ and  $H$ is a graph with $v(H)=k$. We have that 
        \begin{numcases}
        {M_H(n)\geq }
        \Omega(n^2) & if  $k\geq 6$ and $\delta(H)> 3k/4$; \label{thm:dense}\\ 
        n^{2-O(1/\sqrt{\log n})}=n^{2-o(1)} & if  $\delta(H)\geq k/2+1$. \label{thm: min deg almost quad}
        \end{numcases}   
    \end{theorem}
These two bounds extend the results of \cite{balogh2019maximum} for $K_k$, $k\geq 6$ and for $K_5$, respectively, to arbitrary infection rules that are dense enough in terms of their minimum degree.   
 Later we will see examples of graphs with minimum degree $k/2$ and running time polynomially separated from quadratic, so the second minimum degree condition is optimal. In fact, as we will see, Theorem \ref{thm;dense quad} is best possible not only in guaranteeing almost quadratic running time, but even just super-linear.

A question that remains open is to fully understand the behaviour of quadratic maximum running time in terms of the minimum degree.

    \begin{prob}
    Determine the smallest constant $c_{\min}>0$, such that  any infection rule $H$ with $v(H)=k$ sufficiently large and  $\delta (H) > c_{\min} k$  has $M_H(n) = \Theta (n^2)$.
    \end{prob}
    
In particular is there a graph $H$ with $\delta(H)\geq k/2+1$ that exhibits a sub-quadratic running time $M_H(n)=o(n^{2})$?

\subsection{Random graphs}
Theorem~\ref{thm;dense quad} already establishes that cliques of size at least $5$ are far from unique in having maximum running time that is (almost) quadratic. Our next result shows that in fact {\em almost all} infection rules on $n$ vertices have quadratic running time. In fact, unlike Theorem~\ref{thm;dense quad}, where the density is a large constant, for  typical infection rules $H$, a very sparse density turns out to be sufficient to cause quadratic behaviour.

    \begin{theorem}[\cite{FMSz3}]\label{thm:random}
    Let $H=G(k,p)$. Then with probability tending to 1 as $k\to\infty$ we have that 
        \begin{numcases}
        {M_H(n)=}
        O(1) & if  $p \leq \left( \frac{1}{2} -o(1)\right) \log k/k$; \label{random 0}\\ 
        \Omega(n^2) & if  $p= \omega(\log k/k)$. \label{random 1}
        \end{numcases}   
    \end{theorem}

Theorem \ref{thm:random} shows a stark phase transition in the behaviour of $H=G(k,p)$ around the connectivity threshold $\log k/k$.   We remark that $M_H(n)$ is not necessarily monotone with respect to $H$ and so there is no reason, a priori, to expect such a threshold type result for $H=G(k,p)$.

It is natural to ask whether this threshold is sharper than presented here. The $0$-statement, asserting that $M_H(n)$ is constant when $p$ is small, 
follows from the fact that an isolated edge occurs a.a.s.\ 
(see for example \cite[Theorem 5.4]{bollobas_random_2001}). The $1$-statement, giving quadratic running time for $p$ much larger than $\log k/k$, is more delicate and in particular, we use that $H=G(k,p)$ is self-stable, i.e. the $H$-process is stabilises on $H$ as the starting graph. For this we apply a result of Kim, Sudakov and Vu \cite{kim2002asymmetry}.  Whilst it is believable that results on the asymmetry of $G(k,p)$ can be strengthened to give that $H=G(k,p)$ is a.a.s.\ self-stable as soon as $p> \log k/k$, we did not find such a result in the literature and it is not immediate how one could prove such a statement. Nonetheless, we believe it is likely that our proof can be pushed to give  $M_H(n)=\Omega(n^2)$ a.a.s.\ already when $p>c_1\log k/k$ for any $c_1>1$. The range $\log k/2k<p<\log k/k$ seems to be the most interesting and it is unclear what to expect for the behaviour of $M_H(n)$ with $H=G(k,p)$. 
\begin{prob}
 Determine the typical behaviour of $M_H(n)$ for the random infection rule $H=G(k,p)$ in the range $\left(\frac{1}{2} + o(1)\right)\frac{\log k}{k}<p=O(\log k/k)$. 
\end{prob}
With the additional condition $p < \frac{\log k}{k}$, 
the infection rule $H$ in the problem will a.a.s.\  have one giant component and some isolated vertices, which can be ignored from the perspective of running times. There will also be vertices of degree one which suggests that our chain constructions will be useless for proving lower bounds on $M_H(n)$ as copies of $H$ minus an edge will be abundant and so no analogue of the key property \ref{(*)} from Lemma \ref{lem:chain lower} can hold. Conversely, as we will see in Section \ref{sec:small degrees}, the existence of degree 1 vertices alone is   not  enough to be able to prove  effective upper bounds on running times. 

It would also be very interesting to determine the typical behaviour of $M_H(n)$ for $H$ being a random $d$-regular graph, which we expect to be (almost) quadratic even for $d\geq 3$. Such results would also be helpful to be combined with the methods of Section~\ref{sec:graph parameters} to generate interesting novel behaviours.

\subsection{Constant average degree}
The intuitive rule of thumb for the possibility of slow running time is that the infection rule $H$ is dense. The rationale behind this is that dense graphs in chains are easier to hide from interfering with other chains due to their complex internal structure, while sparse graphs have much less such connections. 
 Theorem \ref{thm:random} about the random graph $H=G(k,p)$  already indicates the limitations of this intuition, showing that a super-logarithmic average degree is enough for a typical infection rule $H$ to have a slow running time.   
Our next  results go further,  identifying two  very sparse explicit families with almost quadratic running time.

\subsubsection{The square of the Hamilton cycle}

The first result is about the $4$-regular family of the {\em squares of Hamilton cycles}. To construct the square of the Hamilton cycle  one orders the $k$ vertices cyclically and connects any two of them whose distance is at most two.

    \begin{theorem}[\cite{FMSz3}] \label{thm: square ham cycle}
    If $H$ with $k\geq 5$ vertices contains the square of a Hamilton cycle, then \[M_H(n)\geq n^{2-O(1/\sqrt{\log n})}=n^{2-o(1)}.\]
    \end{theorem}

Note that Theorem \ref{thm: square ham cycle} also generalises the result on $K_5$ of Balogh et al.~\cite{balogh2019maximum} as the square of a cycle of length 5 is $K_5$. The fact that the result only requires that the infection rule {\em contains} something is noteworthy. As we will see in Section~\ref{sec:graph parameters} the  function $M_H(n)$ is very far from being monotone in the infection rule $H$.

\subsubsection{The wheel} \label{sec:wheel}
Our second result establishes an almost quadratic lower bound for the maximum running time of a sequence of
 graphs with average degree strictly less than four.
   The \textit{wheel graph} $W_k$ is defined to be the graph with $k+1$ vertices   obtained by taking a cycle of length $k$ and adding a new vertex that is adjacent to all vertices of  the cycle.
The striking feature of this example is that we can also couple the lower bound with a non-trivial upper bound. In particular this introduces a completely new behaviour of possible maximum running times, which is sub-quadratic but larger than any polynomial of degree smaller than $2$. This can be seen as evidence for the validity of Conjecture \ref{conj:K5} which predicts a similar behaviour for $K_5$.

    \begin{theorem}[\cite{FMSz3}]\label{thm:wheel}
    Let $k\geq 7$ be an odd integer.
    The wheel graph $W_k$ satisfies \[M_{W_k}(n) \geq n^{2-o(1)} \quad \mbox{ and } \quad  M_{W_k}(n) = o(n^2).\]
    \end{theorem}

In general, upper bounds on $M_H(n)$ are hard to come by and require proving that the process stabilises in a bounded number of rounds \textit{no matter which} $n$-vertex starting graph $G_0$ is given.  One key idea in our  proof  is  that chains $(H_i,e_i)_{i\in [\tau]}$, which are  incredibly useful for lower bounds (see Section \ref{sec:chainsinsection5} below), can also be used for upper bounds! Indeed, when proving upper bounds, we have very little information about  the graphs $G_i$ in the process whose running time we are trying to bound.
 However in any $H$-process $(G_i)_{i\geq 0}$ with running time $\tau:=\tau_H(G_0)$ we can embed an $H$-chain $(H_i,e_i)_{i\in [\tau]}$. Indeed,  starting with some edge $e_\tau$ added at time $\tau$ there is  a copy $H_\tau$ of $H$ completed by $e_\tau$. The edge $e_{\tau-1}$ is then chosen as an edge of $H_\tau$ that is added at time $\tau-1$, which must exist as otherwise $e_\tau$ would be added before time $\tau$. Defining $H_{\tau-1}$ to be a copy of $H$ completed by $e_{\tau-1}$ and continuing this process, we can define the full chain $(H_i,e_i)_{i\in [\tau]}$. 
 Although the starting graph of this chain may not be $G_0$, or even contained in $G_0$, the chain still gives us some concrete structure in which to work on for upper bounds.

 Another interesting feature of our  proof of the upper bound in Theorem \ref{thm:wheel} is that we use the famous \textit{Ruzsa-Szemer\'edi $(6,3)$-problem/theorem} \cite{ruzsa1978triple} which states that  any $n$-vertex $3$-uniform hypergraph with no $6$ vertices inducing 3 or more edges, must have $o(n^2)$ edges.

\subsection{Proof technique: Chain constructions}\label{sec:chainsinsection5}

All of  the lower bounds stated above  as well as several further results later in the survey follow from various chain constructions. Part of the challenge for applying chain constructions to infection rules $H$ that are not complete is to find the ``correct'' way to  generalise definitions and properties. In \cite{FMSz3}, we axiomatised the procedure, giving weak conditions on $H$-chains $(H_i,e_i)_{i\in [\tau]}$ that allow us to control the running time of the process based on a chain. The resulting notion of \textit{proper} chains is somewhat technical. Suffice it to say here that we derive an analogue to Lemma \ref{lem:chain lower} which gives lower bounds on running times for proper chains, with the key property being analogous to the condition \ref{(*)} asking to avoid ``unwanted'' copies $F$ of $H$ minus an edge, that span several $H_i$.   We also placed the linking chains that we saw for $K_5$ into this general framework, which allowed us to simply exhibit some collection $\cH$ of simple $H$-chains which satisfies the notion of properness. In particular, we need to avoid ``local'' unwanted copies $F$ of $H$ minus an edge, that is, those that are induced by one segment chain $(H^a_i,e^a_i)_{i\in[\tau_0]}$ in the collection $\cH$ but do not lie in a single $V(H_i)$, and ``crossing'' copies $F$ of $H$ minus an edge, that span the edges of several different chains in $\cH$.

Avoiding local copies $F$ of $H-e$ for some edge $e\in H$ essentially boils down to 
showing that we cannot place such a copy $F$ on the edges of a simple $H$-chain $(H_i,e_i)_{i\in[\tau_0]}$ without placing it on $V(H_i)$ for some $i\in [\tau_0]$. For $H=K_5$, we noticed this via Observation \ref{obs:simple K5}. For general infection rules $H$, there is no guarantee that  we cannot separate a copy $F$ of  $H-e$ across a vertex cut of size $2$ (given by some edge $e_i$ of the chain). This  motivates  the definition of $(2,1)$-inseparability from Section \ref{sec:H perc}, which provides a simple condition on $H$ that suffices to avoid the local unwanted copies of $H-e$.

 As with $K_5$, in many cases we can forbid the presence of crossing unwanted copies $F$ of $H-e$ by placing \textit{dilation chains} $\cH$  according to an appropriate dilation set $A$ which avoids solutions to certain linear equations and placing one simple $H$-chain corresponding to each dilation. Through Fact \ref{obs: K5 col}, we showed that an unwanted copy of $K_5^-$  leads to a triangle spanning more than one chain. In some of our applications in this section, for example for the wheel graph $W_k$, we cannot  guarantee this. We can however work with a slightly weaker concept.

    \begin{Def} \label{def:behre}
    A graph $F$ is called {\em{Behrendian}} if for any non-monochromatic colouring of the edges in $F$ there is a non-monochromatic cycle $C\subseteq F$ which is the union of two or three monochromatic paths.  
    \end{Def}

    The non-monochromatic cycles from Definition \ref{def:behre} can still be avoided by taking a dilation set $A$ which is of Behrend type, as in Theorem \ref{lem:beh stre}, and choosing the constant bounding the size of the coefficients of the equations to be sufficiently large in terms of $k=v(H)$. We therefore arrive at the following theorem.

     \begin{theorem}[\cite{FMSz3}] \label{thm:almostquadratic} 
    Let $H$ be a $(2,1)$-inseparable graph such that $H-e$ is Behrendian for every $e\in H$. Then  $M_H(n) \geq n^{2-o(1)}$. 
    \end{theorem}

Theorem~\ref{thm: square ham cycle} and the second part of Theorem~\ref{thm: min deg almost quad} are  corollaries although checking both the $(2,1)$-inseparability and  Behrendian property is not immediate.  The wheel does not follow from Theorem \ref{thm:almostquadratic} as it is $(2,1)$-separable. Indeed, it has vertices of degree three that can be disconnected from the rest of the graph by removing an incident edge and the two other adjacent vertices. Nonetheless, we can apply our machinery by explicitly building simple $W_k$-chains 
 that avoid local unwanted copies of the wheel minus an edge.

 For the quadratic lower  bounds given in this section, we use \textit{ladder chains} similar to those used for $K_6$, see Figure \ref{fig:ladderchain}. We split the vertex set $V(H)$ into $L\cup R$, place a sequence of copies of  $H[L]$ on the left with consecutive copies intersecting in a single vertex and similarly on the right. Each simple $H$-chain $(H^a_i,e^a_i)_{i\in [\tau]}$ in our collection will then be placed with $H_i$ choosing a copy of  $H[L]$ from the left and a copy of $H[R]$ according to some \textit{slope} $a$. For $K_6$, both $H[L]$ and $H[R]$ were simply copies of $K_3$ and this greatly simplifies the analysis of this construction. Indeed, it is not hard to show that any unwanted copy of $K_6^-$ will have to use one of the triangles on the left and one on the right. To do this for our applications of dense graphs and random graphs is considerably more intricate. Indeed, for the dense graphs we have very little information about the structure of $H$ and  the random graphs $H=G(k,p)$ that we treat can be very sparse. 
 
 For the  first part of Theorem~\ref{thm: min deg almost quad}, as with Theorem \ref{thm:almostquadratic} above, we actually derive a more general theorem. We say that a graph $H$ is {\em $(\ell,1)$-inseparable} if it cannot be disconnected by the removal of an edge and at most $\ell$ vertices.

    \begin{theorem}[\cite{FMSz3}] \label{thm:quadratic}
    If $H$ is  $(\lceil v(H)/2\rceil, 1)$-inseparable, then $M_H(n)=\Omega (n^2)$. 
    \end{theorem}

Further natural examples of $k$-vertex $(\lceil k/2\rceil, 1)$-inseparable graphs are complete tripartite graphs with no part larger than $\floor{k/2}-2$ and $t^{\text{th}}$ powers of cycles for $t\geq k/4+1$, which are obtained by cyclically ordering the vertices and connecting each vertex to all vertices within distance $t$ in the ordering.

\section{Sub-linear running times} \label{sec:sublin}
After seeing infection rules with long running times, we now jump to the other end of the spectrum and explore infection rules for which the process stabilises fast. 

\subsection{Trees} \label{sec:intro-trees}

We begin with an example. 
    \begin{example} For the path $P_t$ with $t\geq 2$ vertices, $M_{P_t}(n)\leq 3$.     \end{example}
 \begin{proof}   If $t=2$, every edge of $K_n$ is infected immediately, so $M_{P_2}(n)=M_{K_2}(n)=1$. 
    More generally, unless the starting graph is $P_t$-stable, the first round creates a copy $P$ of $P_t$. 
    In the second round, both leaves of $P$ become adjacent to every other vertex (including vertices of $P$). Moreover, as at this point every vertex in $V(G_0)$ is a leaf of a copy of $P_t$, in the third round all these vertices also become universal and the process terminates with $K_n$.  
\end{proof}

Infection rules $H$ with a pendant vertex seem to spread the virus very fast:
Any vertex in a copy of $H$ with a pendant neighbour in $H$ will become almost universal, that is, connected to every vertex outside the copy of $H$, in the next round.  
In particular in every tree $T$-process, after just the second round every leaf-neighbour within any copy of $T$ created in the first round is already almost universal. 

Does this mean that there is an absolute constant that bounds the maximum running time for every tree? This is not quite true;  
one can witness this with the tree having just a single leaf-neighbour: the star.

   \begin{example} \label{ex:star}
    Let $T=K_{1,t-1}$ be the $t$-vertex star. Then $M_T(n)=t-1$ for all $n$ sufficiently large. 
    \end{example}
    
    \begin{proof}
Let us first see that if a $T$-process goes on for $t-2$ rounds then it stabilises in the next round. Observe that the centre vertex $c_i$ of a copy of $K_{1,t-1}$ created in the $i^{\text{th}}$ round for some $i\leq t-2$, becomes universal at round $i$. This means that after round $t-2$ every vertex has at least $t-2$ neighbours. So all non-universal vertices are now centres of $K_{1,t-1}$ with one edge missing, and in the next round  all vertices become universal and the $T$-process stabilises by percolating. 

For the lower bound let us take $G_0$ to be the disjoint union of the stars $K_{1,s}$ with centre $c_i$, $1\leq s \leq t-2$ and a set $W$ of $n-\binom{t}2+1$ isolated vertices. We aim to show that the $T$-process $(G_i)_{i\geq 0}$ on $G$ runs for $t-1$ steps before stabilising, that is, at least one edge is infected  
in round $t-1$ (and not before). 
We can see by induction that for $1\leq i \leq t-3$, the graph $G_i$ after round $i$ is the union of the starting graph $G_0$ and all incident edges to the $i$ vertices $c_{t-i-1}, \ldots , c_{t-2}$. Therefore in $G_{t-3}$ the vertices of $W$ are still independent and have degree $t-3$. 
Vertex $c_1$ becomes universal in round $t-2$, raising the degree of the vertices in $W$ to $t-2$. Then in round $t-1$ all edges within $W$ will get infected. That means the process did not stabilise before round $t-1$, so  $\tau_H(G_0) \geq t-1$, as desired.
    \end{proof}

This example shows that, unlike for the case of paths, the maximum running time does depend on the size of the star. Nevertheless it does not depend on the order $n$ of the starting graph and this is what turns out to extend to any tree $T$. Our upper bound gives a quadratic dependence on the number of vertices of $T$.

    \begin{theorem}[\cite{FMSz2}] \label{thm:trees}
    Let $t\in\N$ and $n\geq 2t$. Every tree $T$ on $t$ vertices satisfies
        \[
        M_T(n) \leq \frac 1 8 \cdot (t^2 + 6t + 68).
        \]
    \end{theorem}
 
As discussed above, a neighbour of a leaf in a copy of the target tree $T$ in $G_1$ becomes an almost universal vertex during the second round, and hence the diameter of $G_2$ already becomes constant, so the statement of Theorem~\ref{thm:trees} perhaps does not come across as unexpected. The actual proof however, that from constant diameter the percolation process  finishes in constant time, turned out to be a considerable challenge. Indeed, the analysis needed to handle all the great variety of trees and choices of starting graph is quite delicate. Moreover, extra ideas were needed to achieve an upper bound whose dependence on  $t$  is polynomial. An interesting question is to investigate how tight the upper bound of Theorem \ref{thm:trees} is.

    \begin{prob}[\cite{FMSz2}]\label{prob:trees}
        \label{q:tree}
    For $t\in \N$, determine $$M^*(t):=\max\{\limsup_{n\rightarrow \infty}M_T(n): T \mbox{ is a } t\mbox{-vertex tree}\}.$$
    \end{prob}

We avoid simply taking a maximum here as it may be that small values of $n$ interfere with the asymptotics and we take the $\limsup$ instead of a limit as we cannot rule out the possibility that $M_T(n)$ oscillates with $n$. 
With this notation, Theorem \ref{thm:trees} gives that $M^*(t)=O(t^2)$ and stars (Example \ref{ex:star})  give $M^*(t)\geq t-1$. 
We  expect that stars give the longest maximum running time.

    \begin{conj}[\cite{FMSz2}] \label{conj:tree}
    For all $t\in \N$, we have $M^*(t)=t-1$. 
    \end{conj}

\subsection{Cycles} 
As discussed in the introduction, the $K_3$-process always finishes in at most $\lceil\log_2 (n-1)\rceil$ rounds, and taking the path $P_n$ to be the starting graph realises this. 
This is still a quite fast  infection spread, but not constant anymore.  With the proof of this in mind it should be of no great surprise that the maximum running time of the $C_k$-process for any cycle length $k\geq 3$ is  of  order $\log_{k-1} n$. Indeed, the key observation again is that in one step of the process the distance between any two vertices along a shortest path between them is essentially reduced to its $1/(k-1)$-fraction. 
The next theorem determines the precise value of $M_{C_k}(n)$ for all $k\geq 3$ and all $n$ sufficiently large, giving the first infinite family of graphs $H$ (other than stars and paths) for which the exact running time $M_H(n)$ has been determined. Indeed, the only other non-trivial cases known is when $H=K_4$ \cite{matzke2015saturation,przykucki2012maximal}.

    \begin{theorem}[\cite{FMSz1}]\label{thm:cycles}
    Let $k\geq 3$. For sufficiently large $n\in\N$ we have
        \begin{equation}\label{eq:cycles_main}
        M_{C_k}(n) = \begin{cases}
        \ceil*{\log_{k-1} (n+k^2-4k+2)} & \text{if $k$ is odd} ;\\
        \ceil*{\log_{k-1}\left(2n+k^2-5k\right)} & \text{if $k$ is even}.
        \end{cases}
        \end{equation}
    \end{theorem}

 What is perhaps unexpected is  that the jumps of the monotone increasing function $M_{C_k}(n)$, i.e. those $n\in\N$ with $M_{C_k}(n+1) =  M_{C_k}(n)+1$, behave differently in terms of $k$ depending on the parity of $k$. For odd $k$ the jumps are close to powers of $k-1$ while for even $k$ the jumps are near one half times powers of $k-1$.  Moreover, the precise location of the jumps is determined by a quadratic function of $k$ in both the even and odd cases. This function turns out to be controlled by the  \textit{Frobenius number} of the numerical semi-group generated by $k-2$ and $k$, that is, the largest  natural number that cannot be expressed as an integral linear combination of $k-2$ and $k$ with non-negative coefficients. 
 
 The difference in behaviour between the odd and even case is also witnessed by the starting graphs that achieve the maximum running time.  Whilst the path still provides the maximum running time $M_{C_k}(n)$ when $k$ is odd, for even $k$ the optimal starting graph is slightly different (see Figure \ref{fig:even_cycles}) and the analysis of the $C_k$-process on this graph is more delicate. 

    \begin{figure}[h]
    \centering
    \includegraphics[width=0.7\linewidth]{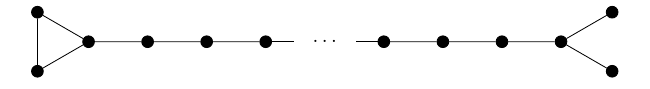}
    \caption{A starting graph maximising the running time of the $C_{k}$-process for $k\in 2\mathbb{N}$.}
    \label{fig:even_cycles}
    \end{figure}
    
 \section{Graph parameters and the running time} \label{sec:graph parameters}
In Theorem \ref{thm:trees} and Theorem~\ref{thm:cycles} we saw that the maximum running time of the $H$-process is constant for trees and logarithmic for cycles, respectively. These are in stark contrast with cliques of order $4$, where the running time of the process could be at least linear, and the wide array of graphs  encountered in Section~\ref{sec:quad} where the running time could be (almost) quadratic. 
Which properties of the infection rule $H$ cause the $H$-process to always finish quickly?  This is what we study in this section, and try to identify necessary and sufficient conditions for sub-linear or linear maximum running time, in terms of relevant graph parameters, involving the degrees, the connectivity and the tree-width of the infection rule. In this section we focus  on connected infection rules (see Section \ref{sec:discon} for the disconnected case). 

\subsection{The role of the degrees} \label{sec:small degrees}
In our analysis of  trees, the existence of leaves plays a crucial role in showing that the running time is constant. In the analysis of the percolation threshold $p_c(H)$~\cite{balogh2012graph} in $G(n,p)$ the presence of a single pendant vertex in the infection rule $H$ also plays a central role in pushing the threshold towards a sparser range  (see Section \ref{sec: H perc thresh}).  
One cannot help but wonder whether the minimum degree being $1$ has anything to do with small maximum running time in general. 
It turns out that, with the exception of cycles, it is indeed \emph{necessary} for a sub-linear running time. 

    \begin{theorem}[\cite{FMSz2}]\label{thm:23bound}
Let $H$ be a connected graph. If $\delta (H) \geq 2$ and $\Delta (H) \geq 3$ then $M_H(n)= \Omega (n)$. 
 \end{theorem}
 \begin{cor}\label{cor:23bound}
 If $H$ is connected and $M_H(n) = o(n)$ then $\delta (H) = 1$ or $H$ is a cycle. 
    \end{cor}

The  natural question arises: Is $\delta (H) = 1$ not only (essentially) necessary, but also sufficient for the $H$-process to stabilise in sub-linear time on an arbitrary starting graph? The following simple example points in this direction. It shows that adding a pendant vertex could reduce the maximum running time from quadratic to constant. 

    \begin{example}[\cite{FMSz2}]\label{ex:cliquependant edge}
    Let $k\geq 3$ and $K_k^+$ be the $(k+1)$-vertex graph formed by taking a clique of size $k$ and adding a pendant edge. Then $M_{K_k^+}(n)\leq 3$.
    \end{example}
    \begin{proof}
    Let $G_0$ be an arbitrary $n$-vertex starting graph. Unless the $K_k^+$-process $(G_i)_{i\geq 0}$ immediately stabilises there is some copy of $K_k^+$ in $G_1$. Let $U\subset V(G)$ be the set of $k$ vertices that form a clique in this copy of $K_k^+$. Then in $G_2$, we have that $U$ and $V(G)\setminus U$ form the two parts of a complete bipartite graph. Thus any vertex in $V(G)\setminus U$ is contained in a $k$-clique and consequently in $G_3$ any edge within $V(G)\setminus U$ is added, making $G_3$ the complete graph. 
    \end{proof}

Somewhat surprisingly, despite all the above,  a pendant vertex is  {\em not} sufficient for a sub-linear or even sub-quadratic maximum running time. 

    \begin{theorem}[\cite{FMSz2}] \label{thm:simulate}
    There is a connected  $H$ with $\delta(H)=1$ and $M_H(n)=\Omega(n^2)$. 
    \end{theorem}

    This implies that in our quest towards a characterisation of connected graphs $H$ with sub-linear running time we need to identify additional conditions besides the minimum degree being $1$. Can limiting the maximum degree help? 
The graph $H$ in Theorem~\ref{thm:simulate} has a vertex of full degree (of degree $15$, see Figure \ref{fig:simulate}). The construction in our next result shows that even limiting the maximum degree to $3$ is not helpful to guarantee sub-linear maximum running time. 

    \begin{theorem}[\cite{FMSz2}]\label{thm:counterexample}
    There exists a connected graph $H$ with minimum degree $\delta(H)=1$ and maximum degree $\Delta(H)=3$ satisfying $M_H(n) = \Omega(n)$.
    \end{theorem}

We remark that the maximum degree condition in the theorem cannot be improved to 2 as paths have constant running time by Theorem \ref{thm:trees}. It would be very interesting to determine in general to what extent the limiting of the maximum degree of the infection rule contributes to small maximum running time. The following questions point in this direction. 

    \begin{qn} \label{qn:maximumdegree}
    Let $H$ be a connected graph with $\delta (H)=1$.
    \begin{itemize}
\item  What is the largest maximum running time that could occur if $\Delta (H) = 3$? Is it true that $M_H(n) = O(n)$?
  \item   What is the smallest possible maximum degree $\Delta(H)$ of $H$ if $M_H(n) = \Omega(n^2)$? 
  \end{itemize}
    \end{qn}

It is possible to slightly modify the proofs of the above Theorems~\ref{thm:simulate} and \ref{thm:counterexample} to show that an additional condition in terms of the average degree of $H$ is of no help in guaranteeing faster running time. 
This modification of Theorem~\ref{thm:counterexample} for example creates infection rules with minimum degree $1$, average degree arbitrarily close to $2$, and yet, quadratic maximum running time. On the other hand, connected infection rules with average degree {\em at most} $2$ are unicyclic. These are trees with exactly one extra edge. It is possible, by a small modification of the proof of Theorem~\ref{thm:trees}, to show that these, with the exception of cycles, have constant maximum running time.

\subsection{Proof technique: Simulation} \label{sec:sim}
Let us consider Theorem \ref{thm:simulate} giving a graph $H$ with $\delta(H)=1$ and $M_H(n)$ quadratic. How to construct an infection rule with these intuitively contradicting properties?
The initial high-level idea is to take an infection rule for which we already know how  to produce a quadratic running time on some starting graph $G$, for example the infection rule $K_6$, and extend it to a larger infection rule $H$ with a pendant vertex, such that running the $H$-process on an appropriate extension $\tilde{G}$ of $G$ we can ensure that the original $K_6$-process is somehow also taking place inside, and hence the $H$-process on $\tilde{G}$ will be at least as long as the $K_6$-process on $G$.

A good plan, a nice plan, easier to dream up though than to actually carry it out. 
The first idea would be just to attach a pendant edge to $K_6$, but this was shot down already in Example~\ref{ex:cliquependant edge} to have maximum running time $3$. 
The main problem with adding a pendant vertex to a ``slow infection rule'' becomes highly visible: the pendant edge is added immediately to any vertex in any copy of $K_6$, creating many universal vertices and a speedy finish.  

To circumvent this we aim to add a larger structure to $K_6$ and add a pendant edge to that in such a way that the neighbour of the pendant vertex will always be ``locked'' into the same vertex of $V(\tilde{G})$. This will partly happen because of the asymmetry we create, partly because of other properties of the graph $H$. The following construction will work.

 \begin{figure}[h]
    \centering
    \includegraphics[width=0.4\linewidth]{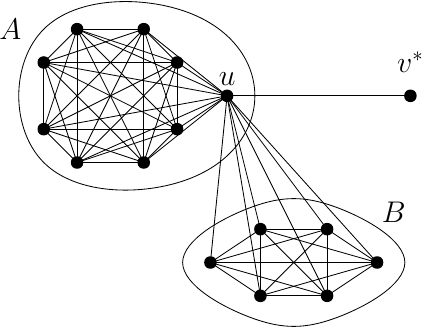}
    \caption{A graph $H$ with minimum degree $1$ and quadratic running time.}
    \label{fig:simulate}
    \end{figure}
The graph $H$ consists of three pairwise vertex-disjoint parts: A clique $B$ of size six, a clique $A$ of size nine  with a special vertex $u$ which is universal, and a pendant vertex $v^*$ attached to it.
As the starting graph $\tilde{G}$ we take the starting graph $G$ of the slow $K_6$-process defined by Balogh et al.~\cite{balogh2019maximum}, together with a disjoint copy $K$ of $K_9$, and make one of its vertices $w$ fully connected to $V:=V(G)$. Then any copy of $K_6-e$ in $V$ during the $H$-process, together with $K$ and an appropriate pendant edge from $w$ back to $V$ forms a copy of $H-e$. So all edges that get infected in the $K_6$-process on $G$ will also get infected in the $H$-process on $\tilde{G}$. We still need to make sure though, that the above described are {\em all} the copies of $H-e$ during the $H$-process on $\tilde{G}$  and hence the process does not speed up compared to the $K_6$-process on $G$, due to an unexpected copy of $H$ minus one of its edges somewhere. The analysis of the $K_6$-process $(G_i)_{i\geq 0}$ given in \cite{balogh2019maximum} implies that no copy of $K_7$ appears for $\Omega (n^2)$ rounds. Using this and an induction over the rounds, one can show that in any copy of $H-e$ the set $A$ has to be mapped to $K$ and $u$ to $w$, as required.

The basic idea of this {\em simulation construction} method is flexible and can be adapted to different settings. The two main ingredients  are some \textit{benchmark} infection rule $B$ (this was $K_6$ above) whose process and running time we intend to simulate  and an appropriate \textit{anchor} graph $A$ (this was $K_9$ above) to which we can connect the benchmark graph $B$ as well as pendent edges to construct $H$.  The starting graph will then be the starting graph for $B$ and a copy of $A$ attached to it. The aim is to show that  $A$ can only appear as itself in the copies of $H$ of the $H$-process so unplanned copies of $H -e$ do not appear. To achieve this we need to construct an anchor graph which is ``far enough'' from anything that will ever appear in the simulated $B$-process on $G$.

For example to prove Theorem \ref{thm:counterexample}, these contemplations motivate us towards the construction of the graph depicted in Figure
\ref{fig:fullcounterexample}.
The  benchmark infection rule $B$ is $C_6$ with a chord which  has at least linear running time by Theorem~\ref{thm:23bound}.
The anchor graph $A$ is a more complex ad hoc construction, together with carefully chosen attachment vertices $u$ and $u'$ (which are now different from each other to control the maximum degree). The proof that this construction works  hinges on an extension~\cite{FMSz2} of Theorem~\ref{thm:23bound} that guarantees that there is some $B$-process  $(G_{i})_{i\geq 0}$ with $\tau_B(G_0)$ linear, such that all of the $G_i$ are bipartite. The crucial property of the anchor graph is  that it contains many triangles, and so is far from bipartite. 

    \begin{figure}[h]
    \centering
    \includegraphics[width=0.7\linewidth]{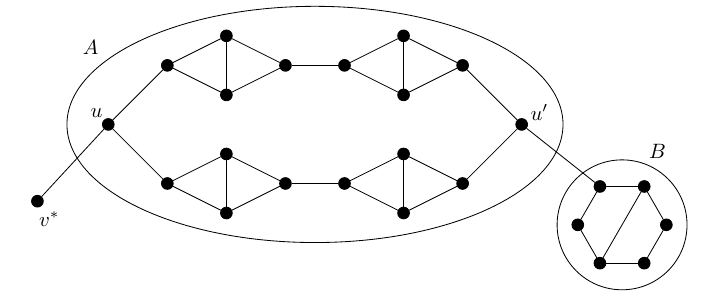}
    \caption{A graph $H$ with $\delta(H)=1$,  $\Delta(H)=3$ and $M_H(n)=\Omega(n)$.}
    \label{fig:fullcounterexample}
    \end{figure}

The small average degree versions of Theorems \ref{thm:simulate} and \ref{thm:counterexample} mentioned at the end of the previous section are obtained
by adding arbitrarily many pendant edges to the neighbour of the pendant vertex in both constructions. The maximum running time does not change since these vertices are locked into their place by the anchor graph already. We discuss some further applications of this simulation construction method in Section \ref{sec:bipartite-rules}.

\subsection{The sub-linear regime}

The characterisation of infection rules with sub-linear maximum running time remains a tantalising open problem.

    \begin{prob}
    Characterise connected graphs $H$ with $M_H(n) = o(n)$. 
    \end{prob}

To answer this question we need a characterisation of graphs $\hat{H}$ with a vertex $u\in V(\hat{H})$, such that for the graph $\hat{H}^{u+}$ obtained from $\hat{H}$ by attaching a pendant edge at vertex $u$, has sub-linear running time.  Even though we do not do not have a precise conjecture for this question, we are confident that the possible maximum running times are quite limited in this range.

    \begin{conj}\label{conj:sublinear}
    If $H$ is connected with $M_H(n) = o(n)$  then $M_H(n) = \Theta(1)$ or $\Theta (\log n)$.
    \end{conj}

To gain some more insight, let us run the $\hat{H}^{u+}$ process on some starting graph $G$. How many rounds can there be when there is a newly infected edge playing the role of the pendant edge in a copy of $\hat{H}^{u+}$? Note that in such a \textit{pendant round} the vertex playing the role of $u$ immediately becomes almost universal, that is, adjacent to all but $v(\hat{H})$ vertices in $V=V(G)$. Now if it so happens that at least $v(\hat{H})^2$ such almost-universal vertices appear at some time $i$, then $G_i$ contains a clique of size $v(\hat{H})$, and consequently the process will run for at most two more rounds due to the presence of the pendant edge in the infection rule. 
This shows that if  there are more than be $v(\hat{H})^3$ pendant rounds then a very fast finish follows. 
Otherwise, the neighbour of the pendant vertex is locked into at most constantly many places, like in our simulation construction where it is locked into one. We would like to think of the process between two of these very few pendant rounds as resembling the $\hat{H}$-process.
Then either $\delta (\hat H) = 1$ and we proceed by induction or 
$\delta (\hat{H}) \geq 2$ and the maximum running time is either logarithmic or linear time by Theorem~\ref{thm:23bound}.  

The following strengthening of Conjecture \ref{conj:sublinear}  follows from the above heuristic. 

\begin{conj}
  For any connected $H$ and $u\in V(H)$ we have $$M_{H^{u+}}(n) = O(1) \mbox{ \ or\   } \quad \Theta(M_H(n)).$$
\end{conj}

Informally one can say that if $\delta(H)=1$ then either one can control the almost universal vertices, in which case they should not affect the asymptotics of the maximum running time, or we cannot, which means that at least $v(\hat{H})^2$ appear quickly in the process, causing it to terminate in constant time.

\subsection{The role of connectivity}\label{sec:connect}

In Theorem~\ref{thm:23bound} we saw that all connected graphs $H$ that have sub-linear running time also have minimum degree at most two.  
When considering graphs with linear running time this condition is no longer necessary, as we know from $K_4$ for example. 
Is there an analogue of this theorem for (at most) linear running time? 
At the end of this subsection we will see that the answer is no, there are graphs with linear maximum running time and arbitrarily large minimum degree. 

It turns out however that the maximum running time being linear does instead imply a limit on the connectivity of $H$. The notion of $(2,1)$-inseparability (Definition \ref{def:inseparable}), which played  an important role in our chain constructions, also turns out to be the crucial notion in separating linear and super-linear maximum running times. The next result is an analogue of Theorem \ref{thm:23bound} for linear running times and shows that $(2,1)$-inseparable infection rules are super-linear in a strong sense.

    \begin{theorem}[\cite{FMSz3}]\label{thm:connectivity}
    For any $(2,1)$-inseparable graph $H$ we have
        \[
        M_H(n) = \Omega\left( n^{1+\tfrac{2}{3v(H)-2}} \right) .
        \]
    In particular if $M_H(n) = O(n)$ then $H$ is $(2,1)$-separable.
    \end{theorem}

This theorem is proved via yet another type of chain construction, namely \textit{line chains}, where we create a family $\cH$ of simple $H$-chains that is proper by placing one chain on each edge of a linear hypergraph $\cH$, that is, a hypergraph in which pairs of edges intersect in at most one vertex (we then link the chains as we did in Section \ref{sec:K5}). The hypergraph $\cH$ we use has  large \textit{girth}, meaning there is no small Berge cycle in $\cH$, which helps to avoid unwanted copies of $H$ minus an edge. For $\cH$ we use bipartite graphs of large girth due to Lazebnik, Ustimenko and Woldar \cite{lazebnik_connectivity_2004}. 

We remark that $(2,1)$-inseparability is not necessary for super-linear running time. Indeed, we saw in Theorem \ref{thm:wheel} that the wheel graph $W_k$, which 
has minimum degree three and hence is $(2,1)$-separable, has almost quadratic running time for any odd $k\geq 7$.  Yet, in \cite{FMSz2} we  give a partial converse to Theorem \ref{thm:connectivity} showing that $H$ being $(2,1)$-separable is also sufficient for the $H$-process to have linear running time, provided $H$ satisfies the extra condition that it \textit{self-percolates}, i.e., the $H$-process percolates starting from $H$ itself. 

    \begin{theorem}[\cite{FMSz2}]\label{thm:girth_construction}
    If $H$ is $(2,1)$-separable and $\final{H}_H=K_{v(H)}$, then
        \[
        M_H(n) = O(n).
        \]
    \end{theorem}
    
The proof adopts the growing clique idea used for $K_4$ as well as the existence of chains in all processes as used in Theorem \ref{thm:wheel} for the wheel. Although the condition $\final{H}_H=K_{v(H)}$ is quite restrictive, Theorem~\ref{thm:girth_construction} can still be used to generate a number of interesting examples. For example the graph $K_5^-$, that is, $K_5$ minus an edge, percolates on its own vertex set and it is $(2,1)$-separable, hence its maximum running time is at most linear. 
This is in contrast to $K_5$ which has (almost) quadratic running time \cite{balogh2019maximum}. We note that from the other direction Theorem~\ref{thm:23bound} implies $M_{K_5^-}(n)= \Theta (n)$. The following is an infinite sequence of dense examples, showing that the minimum degree condition $k/2 +1$ in Theorem~\ref{thm: min deg almost quad} is best possible to guarantee not only almost quadratic, but also super-linear running times.

    \begin{figure}[h]
    \centering
    \includegraphics[width=0.4\linewidth]{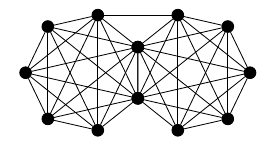}
    \caption{An illustration of the graph $H'_7$.  }
    \label{fig:connectivity}
    \end{figure}

    \begin{example} \label{ex:H'k}
    Let $k\geq 3$ and  $H'_k$ be the $(2k-2)$-vertex graph composed by `gluing together' two cliques of size $k$ along a singular edge $e$ and adding one more edge $e'$ between two non-adjacent vertices. (An illustration of $H'_7$ is given in Figure \ref{fig:connectivity}).    Then $M_{H'_k}(n)=\Theta(n)$. 
    \end{example}

    \begin{proof}
    The $H'_k$-process on $H'_k$ results in $K_{2k-2}$ after just one step of the process as any missing edge can play the role of $e'$ in a copy of $H'_k$. Note, moreover that $\kappa(H'_k-e')=2$ as removing the vertices of the edge $e$ disconnects $H'_k-e'$. Therefore $H'_k$ satisfies the assumptions of Theorem \ref{thm:girth_construction} and so  indeed $M_{H'_k}(n)=\Theta(n)$ where we used Theorem \ref{thm:23bound} for the lower bound on running time here. 
    \end{proof}

\subsection{A conjecture on tree-width}\label{sec:tree-width}
Given Theorem \ref{thm:trees}, another natural parameter to consider is the \emph{tree-width} $\tw(H)$ of the infection rule $H$. It is well known that $\tw (H) \geq \delta (H) \geq \kappa (H)$. Example \ref{ex:cliquependant edge} shows that unlike small minimum degree and small connectivity, small tree-width is \emph{not necessary} for small maximum running time, not even constant. Indeed, for the infection rule $K_k^+$ with maximum running time $3$ we have $\tw(K_k^+)=\tw(K_k)=k-1$. However, it might still be the case that small tree-width is {\em sufficient} to imply an upper bound on maximum running time. Indeed, trees are the unique graphs $H$ with $\tw(H)=1$ and thus Theorem \ref{thm:trees} is simply saying that any graph with tree-width 1 has constant running time.  Graphs of tree-width 2 can already have linear running time, by considering the complete bipartite graphs $K_{2,s}$ and noting that $\tw(K_{k,\ell})=\min\{k,\ell\}$. 

    \begin{prop}[\cite{FMSz2}]\label{prop:k2s_upper}
    For every $s\geq 3$, $M_{K_{2,s}}(n) = \Theta(n)$.
    \end{prop}  

As the minimum degree of $K_{2,s}$ is two and for $s\geq 3$ the maximum degree is at least three, the lower bound on $M_{K_{2,s}}(n)$ follows from Theorem \ref{thm:23bound}. The upper bound in turn establishes that Theorem \ref{thm:23bound} is best possible in the sense that no super-linear maximum running time could be guaranteed in general for graphs with $\delta (H) =2$ and $\Delta (H) \geq 3$. The upper bound for Proposition \ref{prop:k2s_upper} used a new idea of a partition of vertices which coarsens every 4 rounds in which the process does not stabilise. We believe that linear is the largest running time that an infection rule with tree-width two can have and pose the following conjecture. 

    \begin{conj} \label{conj:tw}
    Any graph $H$ with $\tw(H)=2$ has $M_H(n)=O(n)$.
    \end{conj}

Finally we remark that using bounds on tree-width to give effective upper bounds on running time only has the hope to work for very small values of tree-width, namely tree-width 1, as in Theorem \ref{thm:trees} and tree-width 2 as in Conjecture \ref{conj:tw}. Indeed, the wheel graph $W_7$ has tree-width 3 and almost quadratic running time by Theorem~\ref{thm:wheel}.

\section{Bipartite infection rules}\label{sec:bipartite-rules}

In the frustrating world of upper bounds our next one is refreshingly simple, yet meaningfully general. Recall that the  \textit{extremal number} (or\textit{ Tur\'an number}) of a graph, denoted $\ex(n,H)$, is the maximum number of edges of an $n$-vertex graph that does not contain $H$ as a subgraph. 

    \begin{prop}[\cite{FMSz3}] \label{prop:bipartite_extremalnumber}
    Let $H$ be a graph with at least two edges. Then 
        \[
        M_H(n) \leq 2\mathrm{ex}(n,H).
        \]
    \end{prop}
    \begin{proof}
    Let $(G_i)_{i\geq 0}$ be an $H$-process with $\tau=\tau_H(G_0)=M_H(n)\geq 2$. For $1\leq i\leq \tau$, let $e_i\in G_i\setminus G_{i-1}$ be an edge added at time $i$. Now consider the graph $G'$ with  \[E(G')=\{e_i:i\in [\tau], i=1 \mbox{ mod }2\}.\]
    We claim that $G'$ has no copy of $H$. Indeed, if there is some copy $F$ of $H$ in $G'$ and $i_*\in [\tau]$ is the maximal index such that $e_{i_*}\in E(F)$, then $F-e_{i_*}\subseteq G_{i_*-2}$ and so $e_{i_*}$ would be added by time $i_*-1$, contradicting the definition of $e_{i_*}$. Thus, $\tau/2\leq e(G')\leq \ex(n,H)$ as required.  
    \end{proof}

As a consequence of Proposition \ref{prop:bipartite_extremalnumber} and the classical K\H{o}v\'ari-S\'os-Tur\'an theorem \cite{khovari_problem_1954}, we obtain that for bipartite graphs $H$, the maximum running time $M_H(n)$ is polynomially separated from quadratic. 

    \begin{cor}\label{cor:extremal_bipartite}
    Let $H$ be a bipartite graph such that the two partite sets of $H$ have size $r$ and $s$, respectively, where $1\leq r\leq s$. Then
        \[
        M_H(n) = O(n^{2-1/r}).
        \]
    \end{cor}

At first sight, the upper bound of Proposition~\ref{prop:bipartite_extremalnumber} seems awfully generous. For it to be tight, we should be able to construct a pair of $H$-free graphs that are both optimally dense up to a constant factor, such that their edges appear one by one in the even and odd rounds of an $H$-process, respectively. Dense $H$-free graphs are notoriously hard to come by for bipartite $H$ \cite{furedi2013history}, so good luck  locating a magical starting graph producing two of them through the $H$-process! 

Indeed, for many graphs, like trees and unicyclic graphs, we already know from Theorems~\ref{thm:trees} and \ref{thm:cycles} that this upper bound is not tight, by far in fact. 
For complete bipartite graphs with a part of size two we can also do much better, as their maximum running time is linear by Proposition~\ref{prop:k2s_upper}, while the Tur\'an number is of  order $n^{3/2}$. Before completely disrespecting Proposition~\ref{prop:bipartite_extremalnumber} though, we should reflect on what happens for non-bipartite graphs. The generous upper bound {\em is} tight \cite{balogh2019maximum} up to a constant factor for cliques of order at least $6$, and in fact also for ``most'' graphs by Theorem~\ref{thm:random}. So maybe we are also in for a surprise for large enough bipartite $H$?

\subsection{Small graphs}
First we settle that the maximum running time of $K_{3,3}$ is significantly higher than that of $K_{2,s}$. This will follow from  an analogue of Theorem~\ref{thm: min deg almost quad}, that provides a general lower bound for  bipartite infection rules with large proportional minimum degree on both sides. 

    \begin{theorem}[\cite{FMSz3}] \label{thm:bip dense}
    Let $3\leq r\leq s$ and suppose $H$ is a bipartite graph with parts $X,Y$ with $|X|=r$, $|Y|=s$ and such that $d(x)\geq \frac{s}{2}+1$ for all $x\in X$ and $d(y)\geq \tfrac{r}{2}+1$ for all $y\in Y$. Then \[M_H(n)\geq n^{3/2-o(1)}.\]
    \end{theorem}

The bound on the minimum degrees is best possible. Indeed, for an arbitrary $k$ divisible by $4$ consider two copies of $K_{k/4,k/4}$ connected with a single edge. This graph is $(2,1)$-separable and  self-percolating and hence has at most linear running time by Theorem~\ref{thm:connectivity}. The minimum degree on both sides of the bipartition is $k/4$, exactly half of the sides. 

The proof of Theorem \ref{thm:bip dense} follows a bipartite adaptation of our dilated chain construction  and employs generalisations of Sidon sets as the set of dilations.  In fact, arithmetic sets with the exact properties and density that we needed were not known to exist and so in \cite{FMSz3} we also adapted ideas of Ruzsa \cite{ruzsa_solving_1993} to construct the necessary sets. Observe that Theorem \ref{thm:bip dense} is applicable for both $K_{3,3}$ and the cube graph $Q_3$. Unfortunately the known upper bounds of the respective Tur\'an numbers do not provide matching exponents.
For the cube $Q_3$ we can remove the $o(1)$ from the exponent of the lower bound in Theorem~\ref{thm:bip dense} using a bipartite version of the ladder chain construction with Sidon sets.

    \begin{theorem}[\cite{FMSz3}]\label{thm:cube}
    For the running time of the cube $Q_3$ we have
  \[\Omega(n^{3/2})\leq M_{Q_3}(n)\leq O(n^{8/5})
  .\] 
    \end{theorem}
    
Here the upper bound follows from Proposition~\ref{prop:bipartite_extremalnumber} and the best known upper bounds on the Tur\'an number of the cube, due to Erd\H{o}s and Simonovits \cite{erdos_extremal_1970}. 
Actually, for $\ex(n,Q_3)$ the best known bounds are the same as for $M_{Q_3}(n)$. 

The Tur\'an number of $K_{3,3}$ is known asymptotically and the exponent is $5/3$ \cite{furedi2013history}, matching the K\H ovari-S\'os-Tur\'an upper bound. 
While the exponent $3/2$ in the lower bound still falls short of the upper bound of Corollary \ref{cor:extremal_bipartite}, 
both for $H=K_{3,3}$ and $Q_3$ we believe it to be essentially tight. 
    
    \begin{conj} For $K_{3,3}$ and $Q_3$ we have maximum running times
    \[   M_{K_{3,3}}(n) = n^{3/2 + o(1)} \qquad \mbox{ and } \qquad
    M_{Q_3}(n)  = n^{3/2 + o(1)}.\]
    \end{conj}
    
Both for $K_{3,3}$ and $Q_3$, it would even be extremely interesting to improve on any of the known bounds, even if only just by a log-factor. 

\begin{rem}
One could use the slow $K_{3,3}$-process in the simulation construction method (Section \ref{sec:sim}) to create graphs with minimum degree $1$, maximum degree $4$ and maximum running time $n^{3/2}$. 
In the proof of Theorem~\ref{thm:counterexample} one simply replaces the benchmark graph ($C_6$ plus a chord) with $K_{3,3}$. 
 The obtained infection rule has a single vertex of degree $4$. In order to get down to maximum degree $3$, one needs to delete an edge from $K_{3,3}$; unfortunately this reduces the maximum running time to linear.
One can also aim to create a larger almost $3$-regular bipartite infection rule randomly and hope that it has a longer running time. We are not convinced though that such a graph would have super-linear running time either.
\end{rem}

\subsection{Large complete bipartite graphs}
 For larger complete bipartite graphs, using bipartite dilation chains together with a probabilistic alteration tweak, we can improve the lower bound  significantly.

    \begin{theorem}[\cite{FMSz3}] \label{thm:complete bip}
    For $3\leq r\leq s$, the maximum running time $M_{K_{r,s}}(n)$ is bounded from below by
        \begin{equation} \label{eq:comp bip lower}
        M_{K_{r,s}}(n) \geq n^{2-\tfrac{1}{r} - \tfrac{1}{s-1}-o(1)}.
        \end{equation}
    \end{theorem}

The bound in Theorem \ref{thm:complete bip} becomes stronger than that of Theorem \ref{thm:bip dense} when $s>\tfrac{3r-2}{r-2}$, in particular for $K_{r,r}$ with $r\geq 5$. With $r$ fixed and $s$ growing, the bound of Theorem \ref{thm:complete bip} approaches the upper bound given by Corollary \ref{cor:extremal_bipartite}.

Recall that before this section we have only encountered infection rules with maximum running times constant, logarithmic, linear and (almost) quadratic. What other running times are possible?
Using Theorem \ref{thm:complete bip}  together with Proposition~\ref{prop:bipartite_extremalnumber}, we can establish an infinite sequence of infection rules $H$, each with a distinct maximum running time strictly between $n^{3/2}$ and quadratic.

    \begin{cor}[\cite{FMSz3}] \label{cor:compl bip lower}
  For every $3\leq r\in \N$ there exists $s_r\in \N$ such that
  \[\omega\left(n^{2-1/({r-1})}\right) \leq  M_{K_{r,s_r}}(n)\leq O(n^{2-1/r}).\]
    \end{cor}

We expect that there are in fact infinitely many exponents $\alpha\in \mathbb{Q}$  possible such that $M_H(n)=\Theta(n^\alpha)$. However due to our lack of understanding of $M_H(n)$, we do not know that the graphs in Corollary \ref{cor:compl bip lower} offer such examples. In fact we do not even know that maximum running times are approximately polynomial. 
\begin{conj} \label{conj:lim}
    For every infection rule $H$, we have that the limit \[\lim_{n\rightarrow \infty} \frac{\log M_H(n)}{\log n}\] exists.
\end{conj}

In general, it is unclear which of our bounds are closer to the truth for bipartite graphs and the picture is complicated by the fact that the extremal number is not known for many bipartite graphs $H$ of interest, see for example \cite{furedi2013history}. Despite the ample warning signs involving the general case for cliques, we gamble that Proposition \ref{prop:bipartite_extremalnumber} is never tight for bipartite graphs. 

    \begin{conj} \label{conj:bip upper}
    For all bipartite graphs $H$, we have that 
   \[M_H(n)=o(\ex(n,H)).\]
    \end{conj}
It could well be that a stronger conjecture is true and there is always a polynomial separation between $M_H(n)$ and $\ex(n,H)$ for all bipartite graphs $H$. In a different direction, for general $H$ we believe the following is true. 
    \begin{conj}\label{con:H-e} For all infection rules $H$ with at least 3 edges, we have
    $$M_H(n) \leq O\left(\min_e \ex(n,H-e)\right).$$ 
    \end{conj}
This conjecture formulates the limitations intrinsic in all our chain constructions, where the edges $e_i$ form a graph free of any copy of $H$ minus an edge. It can also be considered a challenge to devise novel constructions beyond the current ideas. If one believes the conjecture of Erd\H os \cite{erdos1967some} about the Tur\'an number of $r$-degenerate graphs then Conjecture \ref{con:H-e} implies for the symmetric complete bipartite graph that $M_{K_{r,r}}(n) \leq n^{2-\frac{1}{r-1} -o(1)}$.

\section{Disconnected infection rules} \label{sec:discon}
Up to this point we have restricted our attention to connected infection rules, mainly because we do not expect to understand the behaviour of a disconnected infection rule $H$ if we have not worked out the individual behaviour of its components, and the interpretation of graph percolation as a virus spreading lends itself to connected $H$. 
Nevertheless the general questions we have studied for connected infection rules are mathematically just as valid when $H$ is disconnected, and we close this survey by discussing this setting.

\subsection{Forests and $2$-regular graphs}
A consequence of the proof of Theorem \ref{thm:trees} is that the $T$-process for any tree $T$ will percolate for every starting graph with at least $2v(T)$ vertices that is not $T$-stable, even if it is just a union of $T$ minus an edge and isolated vertices.
Together with the constant maximum running time for trees this implies that $M_F(n) = O(1)$ for any forest $F$.
When studying maximum running times of graph bootstrap processes we are usually interested in starting graphs whose order is much larger than the order of the infection rule.
For forests, investigating the $F$-process on graphs with $v(F)$ vertices is of interest due to connections with Problem \ref{prob:trees}, 
asking for the largest possible running time for $t$-vertex trees.
In particular, an upper bound of the form $M_F(f) \leq c\cdot f$ for Question \ref{q:forest} below would result in a linear upper bound for trees.
We refer to \cite{FMSz2} for details. 

    \begin{qn} \label{q:forest}
    For $f\in \N$, what is  the maximum value of $M_F(f)$ over all  $f$-vertex forests $F$?
    \end{qn}

As for unions of cycles, i.e.\ $2$-regular graphs, recall that when the infection rule is $C_k$ any two vertices of distance $k-1$ become adjacent in each round of the process, which led to the maximum running time on $n$-vertex graphs essentially being $\log_{k-1}(n)$.
In the case that $H$ is the disjoint union of cycles the maximum running time is controlled by the cycle that provides the fastest reduction of the diameter, which is just the largest cycle in $H$.

    \begin{theorem}[\cite{FMSz2}]\label{thm:multiple_cycles}
    If  $s\geq 2$ and $H := C_{k_1} \sqcup \ldots \sqcup C_{k_s}$ is the disjoint union of cycles of lengths $k_1 \geq \ldots \geq k_s$, then for sufficiently large $n$ we have that
        \[
        \log_{k_1-1}(n) - 1 \leq M_H(n) < \log_{k_1-1}(n) + k_1^3s^4.
        \]
    \end{theorem}

The idea underlying the proof is to take a copy of $C_{k_1} \sqcup \ldots \sqcup C_{k_s}$ that appears in the first step of the process, fix all cycles of that copy but the largest one, and consider the $C_{k_1}$-process on the part of the host graph that avoids the fixed cycles. 
This approach is an instance of a more general strategy that we will explore next.

\subsection{Bounds obtained from simulation}
For the rest of our discussion, we concentrate on disconnected infection rules with just two connected components.
Given a graph $H = H_1 \sqcup H_2$, which is the disjoint union of $H_1$ and $H_2$ to what extent is the running time $M_H(n)$ determined by $M_{H_1}(n)$ and $M_{H_2}(n)$?
In particular, can we hope for some non-trivial bounds in terms of $M_{H_1}(n)$ and $M_{H_2}(n)$?

A general strategy to obtain lower bounds on $M_{H_1\sqcup H_2}(n)$ is to take a starting graph $G$ for $H_1$ together with an isolated copy of $H_2$ such that $\final{G}_{H_1}$ does not contain any copies of $H_2$ minus an edge.
In the $H$-process on $G\sqcup H_2$ each completed copy of $H$ will then consist of the fixed copy of $H_2$ and a new copy of $H_1$ with vertices in $V(G)$.
This is essentially the simulation construction introduced in Section \ref{sec:graph parameters}.
For example, consider $H_1 = K_6$ and $H_2$ a sufficiently large wheel, say $H_2 = W_{101}$, with the starting graph $G$ constructed in \cite{balogh2019maximum} to give $M_{K_6}(n) = \Theta(n^2)$.
It is not immediately obvious, but 
possible to verify, that the final graph $\final{G}_{K_6}$ is free of copies of $W_{101}$ minus an edge.
Therefore $\tau_{K_6\sqcup W_{101}}(G\sqcup W_{101}) = \tau_{K_6}(G) = \Theta(n^2) = \omega(M_{W_{101}}(n))$.
In this example $M_{H_1\sqcup H_2}(n)$ has the same asymptotic growth as the larger of $M_{H_1}(n)$ and $M_{H_2}(n)$.

Our next example (Figure \ref{fig:disconnected_2}) demonstrates that $M_H(n)$ may lie strictly between $M_{H_1}(n)$ and $M_{H_2}(n)$.

    \begin{figure}[h]
    \centering
    \includegraphics[width=0.5\linewidth]{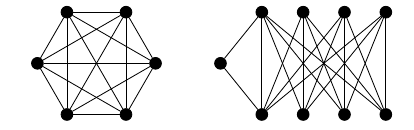}
    \caption{A graph $H = H_1 \sqcup H_2$ whose running time lies strictly between those of its components.}
    \label{fig:disconnected_2}
    \end{figure}

As in our previous example we choose $H_1=K_6$.
The component $H_2$ is obtained by starting with $K_{4,4}$ and adding a new vertex with precisely one neighbour in each part.
Theorem \ref{thm:23bound} implies $M_{H_2}(n) = \Omega(n)$, and it is not difficult to check that due to the vertex of degree two and the symmetries of $K_{4,4}$ it is $(2,1)$-separable and self-percolating, so $M_{H_2}(n) = O(n)$ by Theorem \ref{thm:girth_construction}. On the other hand, $M_{H_1}(n) = \Theta(n^2)$. 
In \cite{FMSz3} it is shown that the running time of $H$ is polynomially separated from those of its components: $n^{1+\alpha}\leq M_H(n)\leq n^{2-\alpha}$ for some $\alpha>0$. The upper bound here uses the fact that once we have more than $\ex(n, K_{4,4})$ edges, we can find many copies of $K_{4,4}$. Once these appear we can show a growing clique behaviour that forces the process to stabilise in at most linearly more steps. The lower bound is obtained via simulation as described above. 
To construct a starting graph $G$ for $K_6$ that is free of $H_2$ minus an edge we use line chains generated by a hypergraph of high girth, similar to the proof of Theorem \ref{thm:connectivity}.

\subsection{Interacting infection rules}
In both examples of the previous section we reduced the disconnected infection rule to a connected one by confining one component to a single isolated copy.
Although in the upper bound for the graph in Figure \ref{fig:disconnected_2} it is crucial that new copies of both components appear throughout the process, we are essentially assuming a two-stage process.
First we consider only new copies of $K_6$, until we can be sure that several copies of $K_{4,4}$ appear, that is, when the number of new edges reaches $\ex(n,K_{4,4})$.
From that point onwards the first component takes over and causes the process to terminate within a linear number of steps.
None of the bounds required any interplay between the two connected infection rules.

An example where the individual infection rules of the components do work together is given in Figure \ref{fig:disconnected_1}:
The triangle component causes vertices of distance two in the host graph to become adjacent at every step (apart from the six vertices that are reserved for the right component), while the right component introduces edges between the endpoints of any two edge-disjoint triangles.
This leads to a maximum running time of $O(1)$. Individually however, the component have maximum running time logarithmic, and linear respectively (using Theorems \ref{thm:23bound} and \ref{thm:girth_construction}). 

    \begin{figure}[h]
    \centering
    \includegraphics[width=0.5\linewidth]{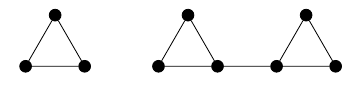}
    \caption{A disconnected graph with maximum running time  $O(1)$.}
    \label{fig:disconnected_1}
    \end{figure}

This example further tells us that Corollary \ref{cor:23bound} breaks down for disconnected infection rules. 
We remark though that Theorem \ref{thm:23bound} can be extended to disconnected $H$ by the stronger requirement that every component of $H$ has minimum degree at least $2$ and maximum degree at least $3$.

We have seen examples where $M_H(n)$ is much smaller than either of $M_{H_1}(n)$ and $M_{H_2}(n)$, asymptotically equal to one of $M_{H_1}(n)$ and $M_{H_2}(n)$, or strictly between the two.
However we are not aware of any choice of $H$ that satisfies $M_H(n) = \omega(M_{H_1}(n))$ and $M_H(n) = \omega(M_{H_2}(n))$.

    \begin{qn}
    Is there a disconnected infection rule $H = H_1 \sqcup H_2$ such that $M_H(n) = \omega(M_{H_1}(n))$ and $M_H(n) = \omega(M_{H_2}(n))$?
    \end{qn}

\subsubsection*{Acknowledgements} The third author is deeply grateful to J\'ozsef Balogh, Gal Kronenberg, and Alexey Pokrovskiy for their collaboration on \cite{balogh2019maximum}, ideas from that paper were instrumental in shaping this survey.  We also thank the anonymous referee for their careful reading and helpful suggestions.

\bibliography{Biblio_short_BCC}

\end{document}